\documentclass{article}

\usepackage{graphicx}

\usepackage{amsmath, amssymb, amsthm, enumerate}

\usepackage{amsmath,amssymb,amsthm,enumerate,mathrsfs}
\newtheorem{theorem}{Theorem}[section]

\newtheorem{lemma}[theorem]{Lemma}
\newtheorem{proposition}[theorem]{Proposition}
\theoremstyle{definition}
\newtheorem{definition}[theorem]{Definition}
\newtheorem{example}[theorem]{Example}

\newtheorem{remark}[theorem]{Remark}

\numberwithin{equation}{section}

\usepackage{color}
\usepackage[mathscr]{euscript}
\usepackage{url}
\usepackage[all]{xy}
\usepackage{booktabs}
\usepackage{epstopdf}
\usepackage{algorithmic}
\usepackage[caption=false]{subfig}
\usepackage{mathtools}
\usepackage[ruled,vlined]{algorithm2e}
\SetAlFnt{\small}
\SetAlCapFnt{\normalsize}
\SetAlCapNameFnt{\normalsize}
\algsetup{linenosize=\small}

\usepackage{hyperref}

\usepackage[capitalize,nameinlink,noabbrev]{cleveref}

\usepackage{tikz}
\usetikzlibrary{shapes}
\usetikzlibrary{arrows}
\usetikzlibrary{matrix}

\definecolor{darkred}{rgb}{0.9,0,0}
\definecolor{darkgreen}{rgb}{0,0.46,0}
\definecolor{purple}{rgb}{0.6,0,0.5}

\newcommand{\pc}{\color{darkred}}
\newcommand{\jl}{\color{purple}}
\newcommand{\fin}{\color{black}}

\newcommand{\tens}[1]{\boldsymbol{\mathcal{#1}}}
\newcommand{\tenselem}[1]{\mathcal{#1}}

\newcommand{\matr}[1]{\boldsymbol{#1}}

\newcommand{\vect}[1]{\boldsymbol{#1}}

\newcommand{\set}[1]{\mathscr{#1}}
\newcommand{\T}{{\sf T}}        

\makeatletter
\newcommand{\rank}[1]{\mathop{\operator@font rank}\{#1\}}
\newcommand{\colrank}[1]{\mathop{\operator@font colrank}\{#1\}}
\newcommand{\krank}[1]{\mathop{\operator@font krank}\{#1\}}
\newcommand{\trace}[1]{\mathop{\operator@font trace}\{#1\}}
\newcommand{\Diag}[1]{\mathop{\operator@font Diag}\{#1\}}    
\newcommand{\diag}[1]{\mathop{\operator@font diag}\{#1\}}    
\newcommand{\Span}[1]{\mathop{\operator@font Span}\{#1\}}    
\newcommand{\argmin}{\mathop{\operator@font argmin}}

\newcommand{\offdiag}[1]{\mathop{\operator@font offdiag}\{#1\}}    
\newcommand{\Proj}[2]{\mathop{\operator@font Proj_{#1}}{#2}}
\newcommand{\ProjGrad}[2]{\mathop{{\operator@font grad}}#1(#2)}
\makeatother

\newcommand{\eqdef}{\stackrel{\sf def}{=}}
\newcommand{\RR}{\mathbb{R}}

\newcommand{\NN}{\mathbb{N}}

\newcommand{\ON}[1]{\set{O}_{#1}}
\newcommand{\SON}[1]{\set{SO}_{#1}}

\newcommand{\Gmat}[3]{\matr{G}^{(#1,#2,#3)}}

\newcommand{\contr}[1]{\mathop{\bullet_{#1}}}   

\DeclareMathOperator*{\argmax}{\arg\!\max}

\makeatletter
\def\bbordermatrix#1{\begingroup \m@th
  \@tempdima 4.75\p@
  \setbox\z@\vbox{%
    \def\cr{\crcr\noalign{\kern2\p@\global\let\cr\endline}}%
    \ialign{$##$\hfil\kern2\p@\kern\@tempdima&\thinspace\hfil$##$\hfil
      &&\quad\hfil$##$\hfil\crcr
      \omit\strut\hfil\crcr\noalign{\kern-\baselineskip}%
      #1\crcr\omit\strut\cr}}%
  \setbox\tw@\vbox{\unvcopy\z@\global\setbox\@ne\lastbox}%
  \setbox\tw@\hbox{\unhbox\@ne\unskip\global\setbox\@ne\lastbox}%
  \setbox\tw@\hbox{$\kern\wd\@ne\kern-\@tempdima\left[\kern-\wd\@ne
    \global\setbox\@ne\vbox{\box\@ne\kern2\p@}%
    \vcenter{\kern-\ht\@ne\unvbox\z@\kern-\baselineskip}\,\right]$}%
  \null\;\vbox{\kern\ht\@ne\box\tw@}\endgroup}
\makeatother

\begin{document}
\makeatletter

\begin{center}
\large{\bf Jacobi-type algorithm for low rank orthogonal approximation of symmetric tensors and its convergence analysis\footnote{This work was supported in part by the National Natural Science Foundation of China (No. 11601371, 11701327) and the  ANR (Agence Nationale de Recherche) under  grant LeaFleT (ANR-19-CE23-0021). }}
\end{center}\vspace{5mm}
\begin{center}

\textsc{Jianze Li\footnote{Shenzhen Research Institute of Big Data, The Chinese University of Hong Kong, Shenzhen, China (lijianze@gmail.com)}, Konstantin Usevich\footnote{Universit\'{e} de Lorraine, CNRS, CRAN, Nancy, France (konstantin.usevich@univ-lorraine.fr)}, and Pierre Comon\footnote{Univ. Grenoble Alpes, CNRS, Grenoble INP, GIPSA-Lab, France (pierre.comon@gipsa-lab.fr)}}\end{center}

\vspace{2mm}

\footnotesize{
\noindent\begin{minipage}{14cm}
{\bf Abstract:}
In this paper,
we propose a Jacobi-type algorithm to solve the low rank orthogonal approximation problem of symmetric tensors.
This algorithm includes as a special case the well-known Jacobi CoM2 algorithm for the approximate orthogonal diagonalization problem of symmetric tensors. 
We study the global convergence of this algorithm under a gradient based ordering for a special case:
the best rank-2 orthogonal approximation of 3rd order symmetric tensors,
and prove that an accumulation point is the unique limit point under some conditions.
We also propose a proximal variant of this algorithm in general case, and prove its global convergence without any further condition.
Numerical experiments are presented to show the efficiency of this algorithm.
\end{minipage}
 \\[5mm]


\hbox to14cm{\hrulefill}\par


\section{Introduction}

As the higher order analogue of vectors and matrices, 
in the last two decades,
tensors have been attracting more and more attentions from various fields,
including signal processing, numerical linear algebra and machine learning
\cite{Cichocki15:review,comon2014tensors,Como10:book,kolda2009tensor,sidiropoulos2017tensor,Anan14:latent}.
One reason is that more and more real data are naturally represented in tensor form, e.g. \emph{hyperspectral images}, \emph{fMRI data}, or \emph{social networks}. 
The other reason is that, compared with the matrix case, tensor based techniques can capture higher order and more complicated relationships, e.g. \emph{Independent Component Analysis} (ICA)  based on the cumulant tensor \cite{Como94:sp}, and \emph{multilinear subspace learning} methods \cite{lu2013multilinear}.

Low rank approximation  of higher order tensors is a very important problem and has been applied in various areas \cite{Como10:book,de1997signal,smildemulti}.  However,
it is much more difficult than the matrix case,
since it is ill-posed for many ranks,
and this ill-posedness is not rare for 3rd order tensors \cite{de2008tensor}.

\textbf{Notation.}
Let $\RR^{n_1\times\cdots\times n_d}\eqdef\RR^{n_1}\otimes\cdots\otimes\RR^{n_d}$ 
be the space of $d$-th order real tensors and
$\text{symm}(\RR^{n\times\cdots\times n})\subset\RR^{n\times\cdots\times n}$ be the set of symmetric ones \cite{Comon08:symmetric,qi2017tensor},
whose entries do not change under any permutation of indices.
The identity matrix of size $n$  is denoted by $\matr{I}_n$.
Let $\text{St}(p,n)\subset\RR^{n\times p}$ be the Stiefel manifold with $1\leq p\leq n$.
Let $\ON{n}\subset\RR^{n\times n}$ be the orthogonal group, \textit{i.e.} $\ON{n}=\text{St}(n,n)$. 
Let $\SON{n}\subset\RR^{n\times n}$ be the special orthogonal group,
\textit{i.e.} 
the set of orthogonal matrices with determinant 1.
We denote by $\|\cdot\|$ the Frobenius norm of a tensor or a matrix,
or the Euclidean norm of a vector.
Tensor arrays, matrices, and vectors,  will be respectively denoted by bold calligraphic letters, e.g. $\tens{A}$, with bold uppercase letters, e.g. $\matr{M}$, and with bold lowercase letters, e.g. $\vect{u}$; corresponding entries will be denoted by $\tenselem{A}_{ijk}$, $M_{ij}$, and $u_i$.
Operator $\contr{p}$ denotes contraction on the $p$th index of a tensor; when contracted with a matrix, it is understood that summation is always performed on the second index of the matrix. For instance, $(\tens{A}\contr{1}\matr{M})_{ijk}=\sum_\ell \tenselem{A}_{\ell jk} M_{i\ell}$.
We denote 
$$\tens{A}(\matr{M}) \eqdef \tens{A} \contr{1} \matr{M}^{\T} \contr{2} \cdots \contr{d} \matr{M}^{\T}$$
for convenience in this paper.
For $\tens{A}\in\RR^{n\times\cdots\times n}$ and a fixed set of indices $1\leq k_1<k_2<\cdots<k_m\leq n$,
we denote by $\tens{A}^{(k_1,k_2,\cdots,k_m)}$ the $m$-dimensional $d$-th order subtensor obtained from  $\tens{A}$ by allowing its indices to take values in $\{k_1,k_2,\cdots,k_m\}$ only.

\textbf{Problem statement.}
Let $\tens{A} \in \text{symm}(\RR^{n\times\cdots\times n})$ and $1 \leq p \leq n$.
In this paper,
we study the \emph{best rank-$p$ orthogonal approximation problem}, 
which is to find
\begin{equation}\label{pro-ortho-approxi-solution}
\tens{C}^{*} \eqdef \sum\limits_{k=1}^{p}\sigma_{k}^{*}u_{k}^{*}\otimes\cdots\otimes u_{k}^{*},
\end{equation}
where
\begin{equation}\label{pro-ortho-approxi}
(\sigma_1^{*},\ldots,\sigma^{*}_p, u^{*}_{1},\ldots,u^{*}_{p}) = \argmin_{\substack{ [u_{1},\cdots,u_{p}]\in \text{St}(p,n) \\ \sigma_{k}\in\RR, 1\leq k\leq p}}\|\tens{A} - \sum_{k=1}^{p}\sigma_{k}u_{k}\otimes\cdots\otimes u_{k}\|.
\end{equation}
If $p=1$,
then \eqref{pro-ortho-approxi} is the \emph{best rank-1 approximation} problem
\cite{Lathauwer00:rank-1approximation,kofidis2002best,kolda2011shifted,zhang2001rank,SilvCA16:spl}
of symmetric tensors,
which is equivalent to the \emph{cubic spherical optimization} problem \cite{qi2009z,zhang2012best,Zhang12:MC}.
If $p=n$, by \cite[Proposition 5.1]{chen2009tensor} and \cite[Proposition 5.2]{LUC2018},
we see that \eqref{pro-ortho-approxi} is closely related to the \emph{approximate orthogonal diagonalization problem} for 3rd and 4th order cumulant tensors,
which is in the core of \emph{Independent Component Analysis} (ICA)  \cite{Como92:elsevier,Como94:sp,comon1994tensor},
and finds many applications \cite{Como10:book}.

To our knowledge,
the orthogonal tensor decomposition was first tackled in \cite{Como92:elsevier}, but appeared more formally in \cite{Kolda01:Orthogonal},
in which many examples were presented to illustrate the difficulties of this type of decomposition.
As shown in  \cite{chen2009tensor}, the minimum in problem \eqref{pro-ortho-approxi} exists, and the \emph{low rank orthogonal approximation of tensors} (LROAT) algorithm and \emph{symmetric LROAT} (SLROAT) were developed to solve this problem based on the polar decomposition.
In the special case $p=1$, the two algorithms boil down to the \emph{higher order power method} (HOPM) and \emph{symmetric HOPM} (SHOPM) algorithm \cite{Lathauwer00:rank-1approximation,kofidis2002best,zhang2001rank}. 
More recently, also based on the polar decomposition, a similar algorithm was developed in \cite{pan2018symmetric} to solve problem \eqref{pro-ortho-approxi}, and this algorithm was applied to the image reconstruction task.

\textbf{Contribution.}
The main contributions of this paper can be summarized as follows: 
\begin{itemize}
\item[$\bullet$] We propose a Jacobi-type algorithm  for solving problem \eqref{pro-ortho-approxi}.
This algorithm is exactly the well-known Jacobi CoM2 algorithm \cite{Como10:book} when $p=n$,
and the same as the Jacobi-type algorithm for Tucker approximation proposed in \cite{IshtAV13:simax} when $p=1$.
\item[$\bullet$] Under the gradient based ordering defined in \cite{IshtAV13:simax,LUC2017globally,ULC2019}, 
we prove the global convergence\footnote{the iterations converge to a unique limit point for any starting point.} of this algorithm for 3rd order tensors of rank $p=2$ under some conditions.
\item[$\bullet$] We propose a proximal variant of this Jacobi-type algorithm, and prove its global convergence without any further condition.
\end{itemize}


\textbf{Organization.}
The paper is organized as follows.
In \cref{geometric}, we show that two optimization problems on Riemannian manifold $\text{St}(p,n)$ and orthogonal group $\ON{n}$ are both equivalent to problem \eqref{pro-ortho-approxi},
and then calculate their Riemannian gradients. 
In \cref{section-algorithm},
we propose a Jacobi-type algorithm to solve problem \eqref{pro-ortho-approxi}.
This algorithm includes the well-known Jacobi CoM2 algorithm as a special case.
In \cref{sect-Jacobi-G},
we study the global convergence of this algorithm under the gradient based ordering for the 3rd order tensor and $p=2$ case.
In \cref{sect-Jacobi-GS}, a proximal variant of this algorithm is proposed, and its global convergence is established. 
In \cref{sect-experiment},
we report some numerical experiments showing the efficiency of this algorithm. 

\section{Geometric properties}\label{geometric}

\subsection{Equivalent problems}
Let $\tens{A} \in \text{symm}(\RR^{n\times\cdots\times n})$ and $1 \leq p \leq n$.
Let $\matr{X}\in \text{St}(p,n)$ and
$\widetilde{\tens{W}} =  \tens{A}(\matr{X}).$
One  problem equivalent to \eqref{pro-ortho-approxi} is to find 
\begin{equation}\label{pro-stefiel}
\matr{X}_{*} = \argmax_{\matr{X}\in \text{St}(p,n)}\tilde{f}(\matr{X}),
\end{equation}
where
\begin{equation}\label{eq-cost-func-1}
\tilde{f}(\matr{X}) \eqdef \sum\limits_{i=1}^{p}\widetilde{\tenselem{W}}^2_{i\cdots i}.
\end{equation}

\begin{lemma}\rm(\cite[Proposition 5.1]{chen2009tensor})
Let $\tens{C}^{*}$
be as in \eqref{pro-ortho-approxi}.
Then
\begin{equation*}\label{eq-orthogonal-A-T}
\langle\tens{A}-\tens{C}^{*}, u_{k}^{*}\otimes\cdots\otimes u_{k}^{*}\rangle = 0\ \ \text{and}\ \
\sigma_{k}^{*} = \langle\tens{A}, u_{k}^{*}\otimes\cdots\otimes u_{k}^{*}\rangle\notag
\end{equation*}
for $1\leq k\leq p$.
Moreover,
it holds that
\begin{equation}\label{eq-relation-problem}
\|\tens{A}-\tens{C}^{*}\|^2 = \|\tens{A}\|^2 - \|\tens{C}^{*}\|^2 = \|\tens{A}\|^2 - \sum\limits_{k=1}^{p}(\sigma_{k}^{*})^2.
\end{equation}
\end{lemma}
%

\begin{remark}\rm\label{remark-equivalence}
(i) Let $\tens{C}^{*}$ be as in \eqref{pro-ortho-approxi}
and $\matr{X}_{*}$ be as in \eqref{pro-stefiel}.
We see from \eqref{eq-relation-problem} that
\begin{equation*}
\matr{X}_{*} = [u_{1}^{*},\cdots,u_{p}^{*}]\ \ \text{and}\ \
\|\tens{A}-\tens{C}^{*}\|^2 = \|\tens{A}\|^2 - \tilde{f}(\matr{X}_{*}).
\end{equation*}
In other words,
to solve \eqref{pro-ortho-approxi},
it is enough for us to solve \eqref{pro-stefiel},
which is an optimization problem on $\text{St}(p,n)$.\\
(ii) If $p=1$, then \eqref{pro-stefiel} is the \emph{cubic spherical optimization problem} \cite{qi2009z,zhang2012best,Zhang12:MC}.
If $p=n$, then \eqref{pro-stefiel} is the \emph{approximate orthogonal tensor diagonalization problem} \cite{Como94:sp,comon1994tensor,Como10:book,LUC2017globally}.
\end{remark}

Let $\matr{Q}\in\ON{n}$ and
$\tens{W} =  \tens{A}(\matr{Q})$.
 Another problem, equivalent to  \eqref{pro-stefiel}, is to find
\begin{equation}\label{pro-orthogonal}
\matr{Q}_{*} = \argmax_{\matr{Q}\in\ON{n}}f(\matr{Q}),
\end{equation}
where
\begin{equation}\label{eq-cost-func-2}
f(\matr{Q})\eqdef \sum\limits_{i=1}^{p}\tenselem{W}^2_{i\cdots i}.
\end{equation}
In fact,
if $\matr{X}\in \text{St}(p,n)$ and $\matr{Q}=[\matr{X},\matr{Y}]\in\ON{n}$,
then $\tenselem{W}_{i_1\cdots i_d}=\widetilde{\tenselem{W}}_{i_1\cdots i_d}$ for any $1\leq i_1,\cdots,i_d\leq p$.
The equivalence between \eqref{pro-stefiel} and \eqref{pro-orthogonal} follows from the fact that $f(\matr{Q}) = \tilde{f}(\matr{X})$.

\begin{remark}\rm
Let $\tens{W}\in \text{symm}(\RR^{n\times n\times n})$ and $1\leq p\leq n$.
Let $\widetilde{\tens{W}}=\tens{W}^{(1,2,\cdots,p)}$.  
Then the objective used in \cite[(3.1)]{IshtAV13:simax} is the sum of squares of all the elements in $\widetilde{\tens{W}}$,
while \eqref{eq-cost-func-2} is the sum of squares of the diagonal elements in $\widetilde{\tens{W}}$.
They are the same if $p=1$.
\end{remark}

\subsection{Riemannian gradient}

\begin{definition}\rm
Let $\tens{A} \in \text{symm}(\RR^{n\times\cdots\times n})$ and $1\leq i<j\leq n$.
Define
\begin{align*}
\sigma_{i,j}(\tens{A})\eqdef \tenselem{A}_{ii\ldots i}\tenselem{A}_{ji\ldots i},\quad
d_{i,j}(\tens{A})\eqdef \sigma_{i,j}(\tens{A})-\sigma_{j,i}(\tens{A}) = \tenselem{A}_{ii\ldots i}\tenselem{A}_{ji\ldots i}-\tenselem{A}_{ij\ldots j}\tenselem{A}_{jj\ldots j}.
\end{align*}
\end{definition}

\begin{theorem}\label{RiemanGrad-thm}\rm
The Riemannian gradient of
\eqref{eq-cost-func-2} at $\matr{Q}$ is
\begin{equation}\label{eq-Riemannian-gradient}
\ProjGrad{f}{\matr{Q}}= \matr{Q}\,\Lambda(\matr{Q}),
\end{equation}
where
\begin{align}\label{eq-gradient-On}
\Lambda(\matr{Q})\eqdef
d\cdot
\left[\begin{smallmatrix}
0 & -d_{1,2}(\tens{W}) &
\ldots & -d_{1,p}(\tens{W})
& -\sigma_{1,p+1}(\tens{W})
& \cdots & -\sigma_{1,n}(\tens{W})\\ \\
d_{1,2}(\tens{W}) & 0 &
\ldots & -d_{2,p}(\tens{W})
& -\sigma_{2,p+1}(\tens{W})
& \cdots & -\sigma_{2,n}(\tens{W})\\ \\
\ldots&\ldots&\ldots&\ldots&\cdots&\cdots&\cdots\\ \\
d_{1,p}(\tens{W}) &
d_{2,p}(\tens{W}) & \ldots & 0
& -\sigma_{p,p+1}(\tens{W})
& \cdots & -\sigma_{p,n}(\tens{W})\\ \\
\sigma_{1,p+1}(\tens{W}) & \sigma_{2,p+1}(\tens{W})
& \cdots &\sigma_{p,p+1}(\tens{W}) & 0 &\cdots & 0\\ \\
\ldots&\ldots&\ldots&\ldots&\ldots&\cdots & \ldots\\ \\
\sigma_{1,n}(\tens{W}) &
\sigma_{2,n}(\tens{W}) & \ldots & \sigma_{p,n}(\tens{W}) &0
&\cdots& 0
\end{smallmatrix}\right].
\end{align}
\end{theorem}

\begin{proof}
Note that
\begin{equation*}
f(\matr{Q}) = \sum\limits_{j=1}^p \tenselem{W}^2_{jj\ldots j}
=\sum\limits_{j=1}^p(\sum\limits_{i_1,i_2,\ldots,i_d}\tenselem{A}_{i_1,i_2,\ldots,i_d}Q_{i_1,j}Q_{i_2,j}\ldots Q_{i_d,j})^2.
\end{equation*}
Let $\tens{V} = \tens{A} \contr{2} \matr{Q}^{\T} \cdots \contr{d} \matr{Q}^{\T}$.
Fix $1\leq i\leq n$ and $1\leq j\leq p$.
Then
\begin{align*}
\frac{\partial f}{\partial {Q}_{i,j}}
=2d\tenselem{W}_{jj\ldots j} \tenselem{V}_{ij\ldots j}
\end{align*}
by methods similar to \cite[Section 4.1]{LUC2017globally}.
Note that $\tens{W} = \tens{V}\contr{1} \matr{Q}^{\T}$.
We get the Euclidean gradient of \eqref{eq-cost-func-2} at $\matr{Q}$ as follows:
\begin{align*}
\nabla f(\matr{Q})
=  2d
\matr{Q}
  \begin{bmatrix}
    \tenselem{W}_{11\ldots 1} & \tenselem{W}_{12\ldots 2} & \ldots & \tenselem{W}_{1p\ldots p}& 0 &\cdots&0\\
    \tenselem{W}_{21\ldots 1} & \tenselem{W}_{22\ldots 2} & \ldots & \tenselem{W}_{2p\ldots p}& 0 &\cdots&0\\
   \ldots&\ldots&\ldots&\ldots& \cdots &\cdots&\cdots\\ \\
   \tenselem{W}_{n1\ldots 1} & \tenselem{W}_{n2\ldots 2} & \ldots & \tenselem{W}_{np\ldots p}& 0 &\cdots&0
  \end{bmatrix}
   \begin{bmatrix}
    \tenselem{W}_{1\ldots 1} & \cdots & 0&\cdots &0\\
    \vdots&\ddots&\vdots&\cdots &0\\
    0 & \cdots & \tenselem{W}_{p\cdots p}&\cdots &0\\
    \vdots&\cdots&\cdots&\ddots &\vdots\\
    0&\cdots&0&\cdots &0
  \end{bmatrix}.
\end{align*}
By \cite[(3.35)]{Absil08:Optimization}, 
we get that 
\begin{equation}\label{eq-Riemannian-gradient}
\ProjGrad{f}{\matr{Q}}=
\frac{1}{2}\matr{Q}(\matr{Q}^{\T}\nabla f(\matr{Q})-\nabla f(\matr{Q})^{\T}\matr{Q})
=\matr{Q}\,\Lambda(\matr{Q}).
\end{equation}
Then the proof is complete. 
\end{proof}

\begin{remark}\rm
(i) If $p=1$,
we see that $\Lambda(\matr{Q})=0$ if and only if
\begin{equation*}
\tenselem{W}_{21\ldots 1} = \tenselem{W}_{31\ldots 1} = \cdots = \tenselem{W}_{n1\ldots 1} = 0,
\end{equation*}
which means that the first column of $\matr{Q}$ satisfies the condition in \cite[(2)]{qi2009z}.\\
(ii) The definition of $\Lambda(\matr{Q})$ in \eqref{eq-gradient-On} can be seen as an extension of \cite[(12)]{LUC2017globally}.
\end{remark}

\begin{theorem}\rm
The Riemannian gradient of \eqref{eq-cost-func-1} at $\matr{X}$ satisfies
\begin{align}\label{Remannian-gradient-stp-2}
\matr{X}^{\T}\ProjGrad{\tilde{f}}{\matr{X}} = d\cdot
\left[\begin{smallmatrix}
0 & -d_{1,2}(\widetilde{\tens{W}}) &
\ldots & -d_{1,p}(\widetilde{\tens{W}})\\ \\
d_{1,2}(\widetilde{\tens{W}}) & 0 &
\ldots & -d_{2,p}(\widetilde{\tens{W}})\\ \\
\ldots&\ldots&\ldots&\ldots\\ \\
d_{1,p}(\widetilde{\tens{W}}) &
d_{2,p}(\widetilde{\tens{W}}) & \ldots & 0
\end{smallmatrix}\right].
\end{align}
\end{theorem}
\begin{proof}
The proof goes along the same lines as for \cref{RiemanGrad-thm}. Note that
\begin{equation*}
\tilde{f}(\matr{X}) = \sum\limits_{j=1}^p \widetilde{\tenselem{W}}^2_{jj\ldots j}
=\sum\limits_{j=1}^p(\sum\limits_{i_1,i_2,\ldots,i_d}\tenselem{A}_{i_1,i_2,\ldots,i_d}X_{i_1,j}X_{i_2,j}\ldots X_{i_d,j})^2.
\end{equation*}
Let $\widetilde{\tens{V}} = \tens{A} \contr{2} \matr{X}^{\T} \cdots \contr{d} \matr{X}^{\T}$.
Fix $1\leq i\leq n$ and $1\leq j\leq p$.
Then
\begin{align*}
\frac{\partial \tilde{f}}{\partial X_{i,j}}
=2d\widetilde{\tenselem{W}}_{jj\ldots j} \widetilde{\tenselem{V}}_{ij\ldots j}
\end{align*}
by the similar methods in \cite[Section 4.1]{LUC2017globally}.
Note that $\widetilde{\tens{W}} = \widetilde{\tens{V}}\contr{1} \matr{X}^{\T}$.
We get the Euclidean gradient of \eqref{eq-cost-func-2} at $\matr{X}$ as follows:
\begin{align*}
\nabla \tilde{f}(\matr{X})&=  2d
  \begin{bmatrix}
    \widetilde{\tenselem{V}}_{11\ldots 1} & \widetilde{\tenselem{V}}_{12\ldots 2} & \cdots & \widetilde{\tenselem{V}}_{1p\ldots p} \\
    \widetilde{\tenselem{V}}_{21\ldots 1} & \widetilde{\tenselem{V}}_{22\ldots 2} & \cdots & \widetilde{\tenselem{V}}_{2p\ldots p} \\
   \cdots&\cdots&\cdots&\cdots\\
   \widetilde{\tenselem{V}}_{n1\ldots 1} & \widetilde{\tenselem{V}}_{n2\ldots 2} & \cdots & \widetilde{\tenselem{V}}_{np\ldots p}
  \end{bmatrix}
    \begin{bmatrix}
    \widetilde{\tenselem{W}}_{1\ldots 1} & \cdots & 0\\
    \vdots & \ddots & \vdots\\
    0 & \cdots & \widetilde{\tenselem{W}}_{p\cdots p}
  \end{bmatrix}.
\end{align*}
It follows by \cite[(3.35)]{Absil08:Optimization} that
\begin{align}\label{Remannian-gradient-stp-1}
\ProjGrad{\tilde{f}}{\matr{X}}
= (\matr{I}_{n}-\matr{X}\matr{X}^{\T})\nabla \tilde{f}(\matr{X})
+ d\matr{X}\cdot
\left[\begin{smallmatrix}
0 & -d_{1,2}(\widetilde{\tens{W}}) &
\ldots & -d_{1,p}(\widetilde{\tens{W}})\\ \\
d_{1,2}(\widetilde{\tens{W}}) & 0 &
\ldots & -d_{2,p}(\widetilde{\tens{W}})\\ \\
\ldots&\ldots&\ldots&\ldots\\ \\
d_{1,p}(\widetilde{\tens{W}}) &
d_{2,p}(\widetilde{\tens{W}}) & \ldots & 0
\end{smallmatrix}\right],
\end{align}
and the proof is complete.
\end{proof}


\begin{proposition}\rm\label{pro-equiva-sationary}
Let $\tens{A} \in \text{symm}(\RR^{n\times\cdots\times n})$ and $1 \leq p \leq n$.
Let $\matr{X}_{*}\in \text{St}(p,n)$ and $\matr{Q}_{*} = [\matr{X}_{*},\matr{Y}_{*}]\in \ON{n}.$
Suppose that $\tilde{f}$ is as in \eqref{eq-cost-func-1} and $f$ is as in \eqref{eq-cost-func-2}.
Then
$$\ProjGrad{\tilde{f}}{\matr{X}_{*}}=0\ \Leftrightarrow \ \ProjGrad{f}{\matr{Q}_{*}}=0.$$
\end{proposition}
\begin{proof}
Let
$\widetilde{\tens{W}}_{*} =  \tens{A}(\matr{X}_{*})$
and
$\tens{W}_{*} =  \tens{A}(\matr{Q}_{*})$.\\
($\Rightarrow$).
By \eqref{Remannian-gradient-stp-2},
we see that
$d_{i,j}(\tens{W}_{*}) = d_{i,j}(\widetilde{\tens{W}}_{*}) = 0$
for any $1\leq i<j \leq p$.
It follows by \eqref{Remannian-gradient-stp-1} that
$$\matr{Y}_{*}\matr{Y}_{*}^{\T}\nabla \tilde{f}(\matr{X}_{*})
= (\matr{I}_{n}-\matr{X}_{*}\matr{X}_{*}^{\T})\nabla \tilde{f}(\matr{X}_{*}) = 0,$$
and thus
$$\matr{Y}_{*}^{\T}\nabla \tilde{f}(\matr{X}_{*})
= \matr{Y}_{*}^{\T}\matr{Y}_{*}\matr{Y}_{*}^{\T}\nabla \tilde{f}(\matr{X}_{*}) = 0.$$
Then
$\sigma_{i,j}(\tens{W}_{*}) = 0$
for any $1\leq i\leq p<j\leq n$,
and thus $\ProjGrad{f}{\matr{Q}_{*}}=0$ by \eqref{eq-gradient-On}.\\
($\Leftarrow$).
By \eqref{eq-gradient-On},
we see that
$d_{i,j}(\widetilde{\tens{W}}_{*}) = d_{i,j}(\tens{W}_{*}) = 0$
for any $1\leq i<j \leq p$.
Note that
$\sigma_{i,j}(\tens{W}_{*}) = 0$
for any $1\leq i\leq p< j\leq n$.
It follows that
$\matr{Y}_{*}^{\T}\nabla \tilde{f}(\matr{X}_{*}) = 0,$
and thus
$$(\matr{I}_{n}-\matr{X}_{*}\matr{X}_{*}^{\T})\nabla \tilde{f}(\matr{X}_{*})
= \matr{Y}_{*}\matr{Y}_{*}^{\T}\nabla \tilde{f}(\matr{X}_{*}) = 0.$$
Then
$\ProjGrad{\tilde{f}}{\matr{X}_{*}}=0$ by \eqref{Remannian-gradient-stp-1}.
The proof is complete. 
\end{proof}

\section{Jacobi low rank orthogonal approximation algorithm}\label{section-algorithm}
\subsection{Algorithm description}
Let $1\leq p\leq n$ and $\set{C}=\{(i,j), 1\leq i< j\leq n, i\leq p\}$. 
We divide $\set{C}$ to be two different subsets
\begin{align*}
\set{C}_{1} \eqdef \{(i,j),\ 1\leq i<j\leq p\}\ \ \text{and}\ \
\set{C}_{2} \eqdef \{(i,j),\ 1\leq i\leq p<j\leq n\}.
\end{align*}
Denote by $\Gmat{i}{j}{\theta}$ the \emph{Givens rotation} matrix,  
\begin{equation*}
\Gmat{i}{j}{\theta} = {\small
\bbordermatrix{
&  & & i& & j && \cr
&1 &       & & & &&\cr
 & &\ddots & & & & \matr{0} &\cr
 i && & \cos \theta & & -\sin \theta&& \cr
 & & & & \ddots & &&\cr
j & & & \sin \theta & & \cos \theta && \cr
&  & \matr{0} & & & & \ddots &\cr
& &       & & & &&1}, }
\end{equation*}
as defined \textit{e.g.}  in \cite[Section 2.2]{LUC2017globally}.
Now we formulate the {\it Jacobi low rank orthogonal approximation} (JLROA) algorithm for
problem \eqref{pro-orthogonal} as in \Cref{al-JLROA}.

\begin{algorithm}
\caption{JLROA algorithm}\label{al-JLROA}
\begin{algorithmic}[1]
\STATE{{\bf Input:} $\tens{A}\in\text{symm}(\RR^{n\times \cdots\times n})$, $1 \leq p \leq n$, a starting point $\matr{Q}_{0}$.}
\STATE{{\bf Output:} Sequence of iterations $\matr{Q}_{k}$.} 
\FOR{$k=1,2,\ldots$ until a stopping criterion is satisfied}
\STATE Choose the pair $(i_k,j_k)\in\set{C}$ according to a pair selection rule.
\STATE Solve $\theta_k^{*}$ that maximizes $h_k(\theta)\eqdef\textit{f}(\matr{Q}_{k-1}\Gmat{i_k}{j_k}{\theta})$.
\STATE Set $\matr{U}_k \eqdef \Gmat{i_k}{j_k}{\theta^{*}_k}$, and update $\matr{Q}_k = \matr{Q}_{k-1} \matr{U}_k$.
\ENDFOR
\end{algorithmic}
\end{algorithm}

\begin{remark}
In JLROA algorithm, one natural way of choosing the index pairs is the following cyclic ordering: 
\begin{equation}\label{partial-cyclic-1}
\begin{split}
&(1,2) \to (1,3) \to \cdots \to (1,n) \to \\
& (2,3) \to \cdots \to (2,n) \to \\
& \cdots  \to (p,p+1) \to \cdots \to (p,n) \to \\
&(1,2) \to (1,3) \to \cdots.
\end{split}
\end{equation}
In this cae, we call JLROA algorithm the \emph{JLROA-C} algorithm. 
\end{remark}

\subsection{Elementary rotations}

Let $\tens{W} = \tens{A}(\matr{Q}_{k-1})$ and $\tens{T}(\theta) = \tens{W}(\Gmat{i_k}{j_k}{\theta})$.
As in JLROA algorithm,
we define
\begin{align}\label{definition-h}
\textit{h}_k:\ [-\frac{\pi}{2},\frac{\pi}{2}]\longrightarrow \RR^+, 
\  \theta \longmapsto \textit{f}\ (\matr{Q}_{k-1}\Gmat{i_k}{j_k}{\theta})=\sum_{i=1}^{p}\tenselem{T}_{i\cdots i}^2(\theta)
\end{align}
where $f$ is as in \eqref{eq-cost-func-2}.
Note that $\Gmat{i_k}{j_k}{\theta}=\Gmat{i_k}{j_k}{\theta+2\pi}$ and
$\tenselem{T}_{i\cdots i}^2(\theta)=\tenselem{T}_{i\cdots i}^2(\theta+\pi)$
for any $\theta\in\RR$ and $1\leq i\leq p$.
We see that $\textit{h}_k$ defined on   the whole $\RR$ is  $\pi$-periodic.  
So it is sufficient to determine $\theta_k^{*}\in [-\pi/2, \pi/2]$ such that
$\textit{h}_k(\theta_k^{*})=\max\limits_{\theta} \textit{h}_k(\theta),$
and  choose $\theta_k^{*}$ with the smallest absolute value among all possible $\theta$. 

Denote by $\overline{\RR}=\RR\cup\{\pm\infty\}$.
Define
\begin{align*}
\tau_k:\ \overline{\RR}\longrightarrow \RR^+,
\ \ x \longmapsto \textit{h}_k(\arctan(x)).
\end{align*}
Let $x=\tan(\theta)\in\overline{\RR}$ and $x^{*}_{k}=\tan(\theta_k^{*})$.
Then
\begin{align*}
\tau_k(x) - \tau_k(0) = \textit{h}_k(\theta) - \textit{h}_k(0)
= \sum_{i=1}^{p}\tenselem{T}_{i\cdots i}^2 - \sum_{i=1}^{p}\tenselem{W}_{i\cdots i}^2.
\end{align*}

\begin{lemma}\label{lemma-derivative-h}\rm
Let $\textit{h}_k$ be as in \eqref{definition-h}.
Then
$\textit{h}_k^{'}(\theta)= -2\Lambda(\matr{Q}_{k-1}\Gmat{i_k}{j_k}{\theta})_{i_k,j_k}.$
\end{lemma}
\begin{proof}
We denote by $\matr{G}(\theta)=\Gmat{i_k}{j_k}{\theta}$ for convenience.
It follows from \eqref{eq-Riemannian-gradient} and the methods similar to \cite[Lemma 5.7]{LUC2017globally} that
\begin{align*}
\textit{h}_k^{'}(\theta)&=\langle\ProjGrad{\textit{f}}{\matr{Q}_{k-1}\matr{G}(\theta)}, \matr{Q}_{k-1}\matr{G}^{'}(\theta)\rangle
=\langle \matr{Q}_{k-1}\matr{G}(\theta)\Lambda(\matr{Q}_{k-1}\matr{G}(\theta)), \matr{Q}_{k-1}\matr{G}^{'}(\theta)\rangle\\
&=\langle\Lambda(\matr{Q}_{k-1}\matr{G}(\theta)), {\matr{G}(\theta)}^{\T}\matr{G}^{'}(\theta)\rangle
= -2\Lambda(\matr{Q}_{k-1}\matr{G}(\theta))_{i_k,j_k}.
\end{align*}
\end{proof}

We distinguish two cases:
\paragraph{$\mathbf{Case}$ $1$:} $(i_k,j_k)\in \set{C}_{1}$. In this case the restricted cost function in \eqref{definition-h} takes the form
\[
\textit{h}_k(\theta) = \tenselem{T}_{i_k\cdots i_k}^2(\theta) + \tenselem{T}_{j_k\cdots j_k}^2(\theta),
\]
i.e., the elementary updates are exactly the same as in the case of orthogonal diagonalization \cite{LUC2017globally}.
In particular, $h_k (\theta)$  also has a period $\pi/2$ by \cite[Section 4.3]{LUC2017globally}.
In other words,
we can choose $\theta^{*}_k\in [-\pi/4, \pi/4]$ to maximize $h_k (\theta)$.
Equivalently,
we can choose $x^{*}_{k}\in[-1,1]$ to maximize $\tau_k(x)$.

\paragraph{$\mathbf{Case}$ $2$:} $(i_k,j_k)\in \set{C}_{2}$. In this case, $\textit{h}_k(\theta)$ is simply 
\[
\textit{h}_k(\theta) = \tenselem{T}_{i_k\cdots i_k}^2(\theta),
\]
which, in general, has a period $\pi$, and is a trigonometric polynomial of degree $d$ in $2\pi$.
However, we can use the fact that
\[
(\tenselem{T}_{i_k\cdots i_k}^2(\theta))' = 2(\tenselem{T}'_{i_k\cdots i_k}(\theta))(\tenselem{T}_{i_k\cdots i_k}(\theta)),
\]
to find the update, which is given by finding roots of a $d$-th order real polynomial.
Note that this subproblem is related to rank-one approximation, and therefore the expressions can be found in \cite[Section 3.5]{Lathauwer00:rank-1approximation} and \cite[Section 3.2]{IshtAV13:simax}.

\subsection{Subproblems}

Let $\tens{A}\in\text{symm}(\RR^{n\times \cdots\times n})$ be of 3rd or 4th order.
In the following subsections (\cref{exam-3rd-order} and \cref{exam-4th-order}) we show the details of how to find $\theta_k^{*}$ in JLROA algorithm. 
The derivations in \cref{exam-3rd-order} and \cref{exam-4th-order} for $(i_k,j_k)\in \set{C}_{1}$ case 
were first formulated in \cite{comon1994tensor},
and can also be found in \cite[Section 6.2]{LUC2017globally}. 
We present them here for convenience.

\subsubsection{3rd order symmetric tensors}\label{exam-3rd-order} \paragraph{$\mathbf{Case}$ $1$:} $(i_k,j_k)\in \set{C}_{1}$.
Take $p\ge2$ and the pair $(1,2)$ (without loss of generality). 
Let
\begin{align*}
a &= 6(\tenselem{W}_{111}\tenselem{W}_{112}-\tenselem{W}_{122}\tenselem{W}_{222}),\\
b &= 6(\tenselem{W}_{111}^2+\tenselem{W}_{222}^2-3\tenselem{W}_{112}^2-3\tenselem{W}_{122}^2
-2\tenselem{W}_{111}\tenselem{W}_{122}-2\tenselem{W}_{112}\tenselem{W}_{222}).
\end{align*}
Then we have that
\begin{align}
\tau_k(x)-\tau_k(0)&=\frac{1}{(1+x^2)^2}(a(x-x^3)-\frac{b}{2}x^2),\label{eq-inc-3}\\
\tau_k^{'}(x) &= \frac{1}{(1+x^2)^3}(a(1-6x^2+x^4)-b(x-x^3)).\notag
\end{align}
Denote by $\xi = x - 1/x$.
Then $\tau_k^{'}(x)=0$ if and only if
\begin{align*}
\Omega(\xi) \eqdef a\xi^2+b\xi-4a = 0.
\end{align*}
Solve $\Omega(\xi) = 0$ for all the real roots $\xi_\ell$.
Then solve
$x^2-\xi_\ell x-1=0$
for all $\ell$ and take the best real root as $x_k^{*}$.
\paragraph{$\mathbf{Case}$ 2:} $(i_k,j_k)\in \set{C}_{2}$.
\jl Take $1\leq p<n$ and \fin the pair (1,$n$) (without loss of generality). 
It holds that 
\begin{align*}
\tenselem{T}_{111}(\theta) &= \tenselem{W}_{111} \cos^3 \theta + 3\tenselem{W}_{11n}\cos^2 \theta\sin \theta
+ 3\tenselem{W}_{1nn}\cos \theta\sin^2 \theta +  \tenselem{W}_{nnn}\sin^3 \theta \\
\\
&= \frac{1}{(1+x^2)^{\frac{3}{2}}} \left( \tenselem{W}_{111} + 3\tenselem{W}_{11n}x
+ 3\tenselem{W}_{1nn}x^2  +  \tenselem{W}_{nnn}x^3 \right),
\end{align*}
and hence 
\begin{align}
\tau_k(x)-\tau_k(0) &
= \tenselem{T}_{111}^{2}(\theta)-\tenselem{W}_{111}^2
=\frac{1}{(1+x^2)^3}[
(\tenselem{W}_{nnn}^2-\tenselem{W}_{111}^2)x^6+(6\tenselem{W}_{1nn}\tenselem{W}_{nnn})x^5\notag\\
&+(-3\tenselem{W}_{111}^2+9\tenselem{W}_{1nn}^2+6\tenselem{W}_{11n}\tenselem{W}_{nnn})x^4
+(18\tenselem{W}_{11n}\tenselem{W}_{1nn}+2\tenselem{W}_{111}\tenselem{W}_{nnn})x^3\notag\\
&+(-3\tenselem{W}_{111}^2+6\tenselem{W}_{1nn}\tenselem{W}_{111}+9\tenselem{W}_{11n}^2)x^2
+(6\tenselem{W}_{111}\tenselem{W}_{11n})x],\label{eq-increasement-order3-2}\\
\tau_k^{'}(x)
&= \frac{6\tenselem{T}_{111}(x)}{(1+x^2)^{5/2}}[-\tenselem{W}_{1nn}x^3+(\tenselem{W}_{nnn}-2\tenselem{W}_{11n})x^2
+(2\tenselem{W}_{1nn}-\tenselem{W}_{111})x+\tenselem{W}_{11n}].\notag
\end{align}
Then we solve
\begin{equation}\label{eq-order-3-station}
-\tenselem{W}_{1nn}x^3+(\tenselem{W}_{nnn}-2\tenselem{W}_{11n})x^2+(2\tenselem{W}_{1nn}-\tenselem{W}_{111})x+\tenselem{W}_{11n} = 0,
\end{equation}
and take $x^{*}_{k}$ to be the best point among these real roots and $\pm\infty$.

\begin{remark}\rm
\eqref{eq-order-3-station} is similar to equations in \cite[Section 3.5]{Lathauwer00:rank-1approximation},
which \pc were \fin for the best rank-1 approximation of a tensor in $\text{symm}(\RR^{2\times 2\times 2})$.
\end{remark}

\subsubsection{4th order symmetric tensors}\label{exam-4th-order}
\paragraph{$\mathbf{Case}$ 1:} $(i_k,j_k)\in \set{C}_{1}$.
Take $p\ge2$ and the pair $(1,2)$ (without loss of generality). 
It holds that
\begin{align*}
&\tau_k(x)-\tau_k(0)
= \tenselem{T}_{1111}^{2}+\tenselem{T}_{2222}^{2}-\tenselem{W}_{1111}^2-\tenselem{W}_{2222}^2\\
&=\frac{1}{(1+x^2)^4}((8\tenselem{W}_{1111}\tenselem{W}_{1112}-8\tenselem{W}_{1222}\tenselem{W}_{2222})(x-x^7)\\
&+(-4\tenselem{W}_{1111}^2 + 12\tenselem{W}_{1122}\tenselem{W}_{1111} + 16\tenselem{W}_{1112}^2 + 16\tenselem{W}_{1222}^2
- 4\tenselem{W}_{2222}^2 + 12\tenselem{W}_{1122}\tenselem{W}_{2222})(x^2+x^6)\\
&+(48\tenselem{W}_{1112}\tenselem{W}_{1122} + 8\tenselem{W}_{1111}\tenselem{W}_{1222}
- 48\tenselem{W}_{1122}\tenselem{W}_{1222} - 8\tenselem{W}_{1112}\tenselem{W}_{2222})(x^3-x^5)\\
&+(- 6\tenselem{W}_{1111}^2 + 4\tenselem{W}_{1111}\tenselem{W}_{2222} + 72\tenselem{W}_{1122}^2 - 6\tenselem{W}_{2222}^2 + 64\tenselem{W}_{1112}\tenselem{W}_{1222})x^4).
\end{align*}
Denote by
\begin{align*}
a &= 8(\tenselem{W}_{1111}\tenselem{W}_{1112}-\tenselem{W}_{1222}\tenselem{W}_{2222});\\
b &=8(\tenselem{W}_{1111}^2-3\tenselem{W}_{1122}\tenselem{W}_{1111}-4\tenselem{W}_{1112}^2
-4\tenselem{W}_{1222}^2+\tenselem{W}_{2222}^2-3\tenselem{W}_{1122}\tenselem{W}_{2222});\\
c &= 8(18\tenselem{W}_{1112}\tenselem{W}_{1122}-7\tenselem{W}_{1111}\tenselem{W}_{1112}+3\tenselem{W}_{1111}\tenselem{W}_{1222}\\
&-18\tenselem{W}_{1122}\tenselem{W}_{1222}-3\tenselem{W}_{1112}\tenselem{W}_{2222}+7\tenselem{W}_{1222}\tenselem{W}_{2222});\\
d &= 8(9\tenselem{W}_{1111}\tenselem{W}_{1122}-32\tenselem{W}_{1112}\tenselem{W}_{1222}-2\tenselem{W}_{1111}\tenselem{W}_{2222}\\
&+9\tenselem{W}_{1122}\tenselem{W}_{2222}+12\tenselem{W}_{1112}^2-36\tenselem{W}_{1122}^2+12\tenselem{W}_{1222}^2);\\
e &= 80(6\tenselem{W}_{1122}\tenselem{W}_{1222}-\tenselem{W}_{1111}\tenselem{W}_{1222}-6\tenselem{W}_{1112}\tenselem{W}_{1122}+\tenselem{W}_{1112}\tenselem{W}_{2222}).
\end{align*}
Then
\begin{align*}
&\tau_k'(x) = \frac{1}{(1+x^2)^5}[a(1+x^8)+b(x^7-x)+c(x^6+x^2)+d(x^5-x^3)+ex^4].
\end{align*}
Denote by $\xi = x - 1/x$.
It follows that $\tau_k^{'}(x)=0$ if and only if
\begin{align*}
\Omega(\xi) \eqdef a\xi^4+b\xi^3+(4a + c)\xi^2+(3b + d)\xi+2a+2c+e = 0.
\end{align*}
Solve $\Omega(\xi) = 0$ for all the real roots $\xi_\ell$.
Then solve
$x^2-\xi_\ell x-1=0$
for all $\ell$ and take the best real root as $x_k^{*}$.
\paragraph{$\mathbf{Case}$ 2:} $(i_k,j_k)\in \set{C}_{2}$.
\jl Take $1\leq p<n$ and the pair (1,$n$) \fin (without loss of generality). 
It holds that 
\begin{align*}
\tenselem{T}_{1111}^{2}(\theta) = \frac{1}{(1+x^2)^{2}} \left( \tenselem{W}_{1111} + 4\tenselem{W}_{111n}+6\tenselem{W}_{11nn}x^2
+ 4\tenselem{W}_{1nnn}x^3  +  \tenselem{W}_{nnnn}x^4 \right),
\end{align*}
and hence 
\begin{align*}
\tau_k(x)-\tau_k(0) &
=\frac{1}{(1+x^2)^4}[
(\tenselem{W}_{nnnn}^2-\tenselem{W}_{1111}^2)x^8
+ (8\tenselem{W}_{1nnn}\tenselem{W}_{nnnn})x^7\\
&+ (- 4\tenselem{W}_{1111}^2 + 16\tenselem{W}_{1nnn}^2 + 12\tenselem{W}_{11nn}\tenselem{W}_{nnnn})x^6
+ (48\tenselem{W}_{11nn}\tenselem{W}_{1nnn} + 8\tenselem{W}_{111n}\tenselem{W}_{nnnn})x^5\\
&+ (- 6\tenselem{W}_{1111}^2 + 2\tenselem{W}_{nnnn}\tenselem{W}_{1111} + 36\tenselem{W}_{11nn}^2 + 32\tenselem{W}_{111n}\tenselem{W}_{1nnn})x^4\\
&+ (48\tenselem{W}_{111n}\tenselem{W}_{11nn} + 8\tenselem{W}_{1111}\tenselem{W}_{1nnn})x^3\\
&+ (- 4\tenselem{W}_{1111}^2 + 12\tenselem{W}_{11nn}\tenselem{W}_{1111} + 16\tenselem{W}_{111n}^2)x^2
+ (8\tenselem{W}_{1111}\tenselem{W}_{111n})x],\\
\tau_k^{'}(x)
&= \frac{-8\tenselem{T}_{1111}}{(1+x^2)^3}
[\tenselem{W}_{1nnn}x^4 +(3\tenselem{W}_{11nn}-\tenselem{W}_{nnnn})x^3
+ (3\tenselem{W}_{111n}-3\tenselem{W}_{1nnn})x^2\\
&+ (\tenselem{W}_{1111}-3\tenselem{W}_{11nn})x-\tenselem{W}_{111n}].
\end{align*}
Then we solve
$$\tenselem{W}_{1nnn}x^4 +(3\tenselem{W}_{11nn}-\tenselem{W}_{nnnn})x^3
+ (3\tenselem{W}_{111n}-3\tenselem{W}_{1nnn})x^2
+ (\tenselem{W}_{1111}-3\tenselem{W}_{11nn})x-\tenselem{W}_{111n} = 0$$
and take $x^{*}_{k}$ to be the best point among these real roots and $\pm\infty$.

\section{JLROA-G algorithm and its convergence}\label{sect-Jacobi-G}

\subsection{JLROA-G algorithm}

Different from the cyclic ordering \eqref{partial-cyclic-1} in JLROA-C,
another pair selection rule of Jacobi-type algorithm based on the Riemannian gradient was proposed in \cite{IshtAV13:simax}.
In this sense,
the pair $(i_k,j_k)$ at each iteration is chosen such that
\begin{equation}\label{eq:pair_selection_gradient}
|h_{k}^{'}(0)| = 2|(\matr{Q}_{k-1}^{\intercal}\ProjGrad{f}{\matr{Q}_{k-1}})_{i_k,j_k}|  \ge \varepsilon \|\ProjGrad{f}{\matr{Q}_{k-1}}\|,
\end{equation}
where  $0<\varepsilon\leq2/n$ is fixed.
By \cite[Lemma 5.2]{IshtAV13:simax} and \cite[Lemma 3.1]{LUC2017globally},
we see that it is always possible to find such a pair if $f$ is differentiable.
In this case, we call it the \emph{JLROA-G} algorithm. 

\begin{algorithm}
\caption{JLROA-G algorithm}\label{alg:jacobi-G}
\begin{algorithmic}[1]
\STATE{{\bf Input:} $\tens{A}\in\text{symm}(\RR^{n\times \cdots\times n})$, $1\leq p\leq n$,
$0<\varepsilon\leq 2/n$, a starting point $\matr{Q}_{0}$.}
\STATE{{\bf Output:} Sequence of iterations $\{\matr{Q}_{k}\}_{k\ge1}$.} 
\FOR{$k=1,2,\ldots$ until a stopping criterion is satisfied}
\STATE Choose a pair $(i_k,j_k)$ satisfying the inequality \eqref{eq:pair_selection_gradient} at $\matr{Q}_{k-1}$.
\STATE Solve $\theta^{*}_{k}$ that maximizes $h_k(\theta)$ defined as in \eqref{definition-h}.
\STATE Set $\matr{U}_k \eqdef \Gmat{i_k}{j_k}{\theta^{*}_k}$, and update $\matr{Q}_k = \matr{Q}_{k-1} \matr{U}_k$.
\ENDFOR
\end{algorithmic}
\end{algorithm}

\begin{remark}\rm\label{remark-local-conv}
(i) By \cite[Theorem 5.4]{IshtAV13:simax} and \cite[Theorem 3.3]{LUC2017globally},
we see that every accumulation point of the iterations in JLROA-G is a stationary point of $f$.\\
(ii) Let $\tens{A}\in\text{symm}(\RR^{n\times n\times n})$ and $p$ = 1.
Then JLROA-G is the same with the Jacobi-type algorithm in \cite{IshtAV13:simax},
which was developed to find the best low multilinear rank approximation of symmetric tensors.
\end{remark}

In this section,
we mainly prove the following result for JLROA-G.
The proof is postponed to 
\cref{subsec-main-proof}.

\begin{theorem}\rm\label{theorem-main-covergence}
Let $\tens{A}\in\text{symm}(\RR^{n\times n\times n})$ with $n\geq 3$. 
Suppose that $p=2$ and $\matr{Q}_{\ast}$ is an accumulation point of JLROA-G satisfying
\begin{align}
&\tenselem{A}(\matr{Q}_{\ast})_{112}^2+\tenselem{A}(\matr{Q}_{\ast})_{122}^2\neq 0,\label{eq-condition-not-zero}\\
&\tenselem{A}(\matr{Q}_{\ast})_{333}\tenselem{A}(\matr{Q}_{\ast})_{444}\cdots\tenselem{A}(\matr{Q}_{\ast})_{nnn}\neq 0.\label{eq-condition-not-zero-2}
\end{align}
Then either $\matr{Q}_{\ast}$ is the unique limit point,
or there exist an infinite number of accumulation points.
\end{theorem}

\subsection{Some lemmas}

\begin{lemma}\rm\label{lemma-double-derivative}
Let $\tens{W}\in\text{symm}(\RR^{2\times 2\times 2})$ and
$\tens{T}=\tens{W}(\Gmat{1}{2}{\arctan x})$
with $x\in\overline{\RR}$.
Define
$\tau:\ \overline{\RR}\rightarrow \RR^+$
sending $x$ to $\tenselem{T}_{111}^2.$
Suppose that $\tenselem{W}_{222}\neq0$ and $\tau(0)=\max\limits_{x\in\overline{\RR}} \tau(x).$
Then\\
\noindent (i) $\tenselem{W}_{111}\neq0$, $\tenselem{W}_{112}=0$,\\
\noindent (ii)  $\tenselem{W}_{111}(2\tenselem{W}_{122}-\tenselem{W}_{111})<0$.
\end{lemma}
\begin{proof}
(i) It is clear that $|\tenselem{W}_{222}|\leq|\tenselem{W}_{111}|$ since $\tau(0)\geq\tau(\pm\infty)$. Then $\tenselem{W}_{111}\neq0$.
Let $\theta = \arctan x$.
We have that
\[
\frac{d\tenselem{T}_{111}}{d\theta}=3\tenselem{T}_{112},\ \
\frac{d\tenselem{T}_{112}}{d\theta}=2\tenselem{T}_{122}-\tenselem{T}_{111}
\]
by straightforward differentiation \cite[Page 10]{LUC2017globally}.
It follows that
\begin{align}
\tau'(x)&=2\tenselem{T}_{111}\frac{d\tenselem{T}_{111}}{d\theta}\frac{d\theta}{dx}=\frac{6\tenselem{T}_{111}\tenselem{T}_{112}}{1+x^2},\label{eq-tau-derivative}\\
\tau''(x)
&=\frac{6}{(1+x^2)^2}(3\tenselem{T}_{112}^2+2\tenselem{T}_{111}\tenselem{T}_{122}
-\tenselem{T}_{111}^2-2\tenselem{T}_{111}\tenselem{T}_{112}x).\label{eq-tau-2-derivative}
\end{align}
Note that $\tau'(0)=0$.
We have $\tenselem{W}_{112}=0$ by \eqref{eq-tau-derivative}.\\
(ii) Note that $\tau''(0)\leq0$.
We have
$2\tenselem{W}_{111}\tenselem{W}_{122}-\tenselem{W}_{111}^2\leq0$
by \eqref{eq-tau-2-derivative}.
To complete the proof, we only need to prove that
$\tau(0)<\max\limits_{x\in\overline{\RR}} \tau(x)$ if
$\tenselem{W}_{111}=1$,\ $\tenselem{W}_{122}=1/2$ and $\tenselem{W}_{222}=\beta\neq0$ without loss of generality.
In fact, it can be verified that
$$\tau(x) = \frac{(1+\frac{3}{2}x^2+\beta x^3)^2}{(1+x^2)^3}$$
in this case, and
$$\max\limits_{x\in\overline{\RR}} \tau(x)\geq\tau(2\beta)=\frac{(1+6\beta^2+8\beta^4)^2}{(1+4\beta^2)^3}>\tau(0)=1.$$
\end{proof}

\begin{definition}\label{re-defi-index}\rm(\cite[Definition 3.11]{LUC2018})
Let $\tens{A}\in\text{symm}(\RR^{n\times n\times n})$ and $1\le i<j\le n$.
Suppose that $\tenselem{A}_{iii}\tenselem{A}_{iij}=\tenselem{A}_{ijj}\tenselem{A}_{jjj}$.
The \emph{stationary diagonal ratio},
denoted by $\gamma_{ij}(\tens{A})$,
is defined as follows.
\[
\gamma_{ij}(\tens{A}) \eqdef
\begin{cases}
0, &  \text{if}\ \tens{A}^{(i,j)}= \mathbf{0};\\
\infty, &    \text{if}\ \tenselem{A}_{iii}=\tenselem{A}_{jjj}=0\quad\text{and}\quad\tenselem{A}^2_{ijj} +\tenselem{A}^2_{iij}\neq0;\\
\end{cases}
\]
	otherwise, $\gamma_{ij}(\tens{A})$ is the {(unique)} number such that
\[
\begin{pmatrix}\tenselem{A}_{ijj} \\ \tenselem{A}_{iij} \end{pmatrix}  = \gamma_{ij}(\tens{A})\begin{pmatrix}\tenselem{A}_{iii}\\\tenselem{A}_{jjj}\end{pmatrix}.
\]
\end{definition}

\begin{lemma}\label{lemma-extreme-state-02}\rm
Let $\tens{W}\in\text{symm}(\RR^{2\times 2\times 2})$ and
$\tens{T}=\tens{W}(\Gmat{1}{2}{\arctan x})$
with $x\in\RR$ and $x\neq0$.
Suppose that $\|\diag{\tens{W}}\|=\|\diag{\tens{T}}\|\neq0$ and
$$\tenselem{W}_{111}\tenselem{W}_{112}=\tenselem{W}_{122}\tenselem{W}_{222},\ \ \tenselem{T}_{111}\tenselem{T}_{112}=\tenselem{T}_{122}\tenselem{T}_{222}.$$
Then $\gamma_{12}(\tens{W}) = \gamma_{12}(\tens{T}) = -1$ or $1/3$.
\end{lemma}

\begin{proof}
Note that $\|\diag{\tens{W}}\|=\|\diag{\tens{T}}\|$ and
$\|\tens{W}\| = \|\tens{T}\|$.
We see that
$|\gamma_{12}(\tens{W})| = |\gamma_{12}(\tens{T})|$.
Let $\tens{T} =\tens{W}(\Gmat{1}{2}{\arctan x})$.
Define
\begin{align*}
\tau:\ \RR \longrightarrow \RR^+, 
\  x \longmapsto \|\diag{\tens{T}}\|^2 = \tenselem{T}_{111}^2+\tenselem{T}_{222}^2.
\end{align*}
Then $\tau(x)=\tau(0)$ by the condition.
It follows by \eqref{eq-inc-3} that
\begin{equation}\label{eq-0-double-derivative}
\tenselem{W}_{111}^2+\tenselem{W}_{222}^2-3\tenselem{W}_{112}^2-3\tenselem{W}_{122}^2
-2\tenselem{W}_{111}\tenselem{W}_{122}-2\tenselem{W}_{112}\tenselem{W}_{222}=0.
\end{equation}
After the substitution of $\tenselem{W}_{122}=\gamma_{12}(\tens{W})\tenselem{W}_{111}$ and $\tenselem{W}_{112}=\gamma_{12}(\tens{W})\tenselem{W}_{222}$ to \eqref{eq-0-double-derivative},
we get that $\gamma_{12}(\tens{W})=-1$ or $1/3$.
Note that $\tens{W}=\tens{T}((\Gmat{1}{2}{\arctan x})^\intercal)$.
We can similarly get that $\gamma_{12}(\tens{T})=-1$ or $1/3$.
\end{proof}

\begin{lemma}\rm\label{lemma-3-dimension-p-2}
Let $\tens{W}\in\text{symm}(\RR^{3\times 3\times 3})$ and
$\tens{T}=\tens{W}(\Gmat{1}{3}{\arctan x})$
with $x\in\overline{\RR}$ and $x\neq0$.
Suppose that $|\tenselem{W}_{111}|=|\tenselem{T}_{111}|>0$
and
$$\tenselem{W}_{111}\tenselem{W}_{112}=\tenselem{W}_{122}\tenselem{W}_{222},\ \
\tenselem{T}_{111}\tenselem{T}_{112}=\tenselem{T}_{122}\tenselem{T}_{222},\ \
\tenselem{W}_{113}=\tenselem{W}_{223}=\tenselem{T}_{113}=\tenselem{T}_{223}=0.$$
Then $\tenselem{W}_{112}=\tenselem{W}_{122}=\tenselem{T}_{112}=\tenselem{T}_{122}=0.$
\end{lemma}
\begin{proof}
It can be verified that
$$-\frac{x}{\sqrt{1+x^2}}\tenselem{W}_{122} = \tenselem{T}_{223} = 0,$$
and thus $\tenselem{W}_{122}=0$.
It follows by the condition that $\tenselem{W}_{112}=0$.
Note that
$\tens{W}=\tens{T}((\Gmat{1}{3}{\arctan x})^{\intercal}).$
We can similarly get that $\tenselem{T}_{112}=\tenselem{T}_{122}=0$.
\end{proof}

\subsection{Proof of \cref{theorem-main-covergence}}\label{subsec-main-proof}

Before proving the main theorem,  we recall a result on global convergence from  \cite{LUC2017globally},
which is the direct consequence of \cite[Theorem 2.3]{SU15:pro}.
\begin{theorem}\rm\label{theorem-convegence-general}({\cite[Corollary 5.4]{LUC2017globally}})
Let $f$ be a real analytic function from $\ON{n}$ to $\RR$.
Suppose that $\{\matr{Q}_k\}_{k=1}^{\infty}\subset\ON{n}$ and, for large enough k,\\
(i) there exists $\sigma>0$ such that
\begin{equation*}\label{condition-coro-KL}
|{f}(\matr{Q}_{k})-{f}(\matr{Q}_{k-1})|\geq \sigma\|\ProjGrad{f}{\matr{Q}_{k-1}}\| \|\matr{Q}_{k}-\matr{Q}_{k-1}\|,
\end{equation*}
(ii) $\ProjGrad{f}{\matr{Q}_{k-1}}=0$ implies that $\matr{Q}_{k}=\matr{Q}_{k-1}$.\\
Then the iterations $\{\matr{Q}_k\}_{k=1}^{\infty}$ converge to a point $\matr{Q}_*\in\ON{n}$.
\end{theorem}
Next, we provide a useful bound on the difference between cost functions at consecutive iterations.
\begin{lemma}\rm\label{lemma-weak-inequality}
Let $\tens{A}\in\text{symm}(\RR^{n\times n\times n})$.
Let $h_{k}(\theta)$ be as in \eqref{definition-h} for $k\in\NN$.
Then there exists $\delta>0$ such that
\begin{equation}\label{eq-lemma-weak-ineuqality}
h_{k}(\theta_k^{*})-h_{k}(0)\geq \delta |h_{k}'(0)|^2
\end{equation}
for any $k\in\NN$ with $(i_k,j_k)\in \set{C}_{2}$ in JLROA-G.
\end{lemma}

\begin{proof}
Let $\tens{W} = \tens{A}(\matr{Q}_{k-1})$ and $\tens{T} = \tens{W}(\Gmat{i_k}{j_k}{\theta})$.
Let $(i,j)=(i_k,j_k)$.
It is clear that $\tenselem{T}_{iii}(\theta)$ is a trigonometric polynomial with a finite degree $n_{0}$
for all the iterations in $\set{C}_{2}$.
By \cite[Theorem 1]{bell2015bernstein},
we see that
\begin{equation*}
\tenselem{T}_{iii}'(0)^2 \leq  n_{0}^2(\|\tenselem{T}_{iii}\|_{\infty}^2-\tenselem{T}_{iii}^2(0)) = n_{0}^2(h_{k}(\theta_k^{*})-h_{k}(0)),
\end{equation*}
when $\theta=0$.
Note that $h_{k}'(0)=2\tenselem{T}_{iii}(0)\tenselem{T}_{iii}'(0)$.
Let $M>0$ such that $|4 n_{0}^2\tenselem{T}_{iii}^2(0)|<M$ for all the iterations in $\set{C}_{2}$.
Then
\begin{equation*}
|h_{k}'(0)|^2 \leq  4 n_{0}^2\tenselem{T}^2_{iii}(0)(h_{k}(\theta_k^{*})-h_{k}(0))
<M(h_{k}(\theta_k^{*})-h_{k}(0)).
\end{equation*}
The proof is complete if we set $\delta=1/M$.
\end{proof}

\begin{remark}\rm
Let $\tens{A}\in\text{symm}(\RR^{n\times n\times n\times n})$ be of 4th order.
By the similar methods,
we can also prove \eqref{eq-lemma-weak-ineuqality}
for pairs in $\set{C}_{1}$,
or pairs in $\set{C}_{2}$.
\end{remark}

\begin{proof}[Proof of \cref{theorem-main-covergence}]
Assume that there exist a finite number of accumulation points,
denoted by $\matr{Q}^{(\ell)}(1\leq\ell\leq N)$.
Then any accumulation point is a stationary point by Remark \ref{remark-local-conv}(i).
In other words,
it holds that $\Lambda(\matr{Q}^{(\ell)})=0$ for all $1\leq\ell\leq N$ by \eqref{eq-Riemannian-gradient}.
Let $\matr{Q}_{\ast}=\matr{Q}^{(1)}$.
Now we prove that $\matr{Q}_{\ast}$ is the unique limit point.\\
\textbf{Step 1.}
We first prove that all the accumulation points satisfy \eqref{eq-condition-not-zero} and \eqref{eq-condition-not-zero-2} if $\matr{Q}_{\ast}$ satisfies them.
Note that the number of accumulation points is finite. 
We can see that any two different accumulation points can be connected by finite combination of the following two possible paths.\\
(a) Take the pair $(1,2)\in\set{C}_{1}$.
If $\{x_k^{\ast},(i_k,j_k)=(1,2)\}$ is finite or converges to 0,
this path doesn't appear and we skip it.
Otherwise,
this set has a nonzero accumulation point $\zeta$ and a subsequence converges to it.
We assume that
$$\{x_k^{\ast},(i_k,j_k)=(1,2)\}\rightarrow\zeta\neq0$$
without loss of generality.
Note that
$\{\matr{Q}_{k-1},(i_k,j_k)=(1,2)\}$ has an accumulation point.
We assume that
$$\{\matr{Q}_{k-1},(i_k,j_k)=(1,2)\}\rightarrow\matr{Q}^{(\ell_1)}$$
without loss of generality.
Then $\matr{Q}^{(\ell_2)}=\matr{Q}^{(\ell_1)}\matr{G}^{(1,2,\arctan\zeta)}$ is another different accumulation point.
It is clear that $\tenselem{A}(\matr{Q}^{(\ell_1)})_{iii}=\tenselem{A}(\matr{Q}^{(\ell_2)})_{iii}$ for $3\leq i\leq n$.
Note that $\tens{A}(\matr{Q}^{(\ell_1)})^{(1,2)}$ and $\tens{A}(\matr{Q}^{(\ell_2)})^{(1,2)}$
satisfy the conditions in Lemma \ref{lemma-extreme-state-02}.
We see that
$$\tenselem{A}(\matr{Q}^{(\ell_1)})_{112}^2+\tenselem{A}(\matr{Q}^{(\ell_1)})_{122}^2\neq0,\
\tenselem{A}(\matr{Q}^{(\ell_2)})_{112}^2+\tenselem{A}(\matr{Q}^{(\ell_2)})_{122}^2\neq0.$$
(b) Take the pair $(1,3)\in\set{C}_{2}$ for example.
Other pairs in $\set{C}_{2}$ are similar.
If $\{x_k^{\ast},(i_k,j_k)=(1,3)\}$ is finite or converges to 0,
this path doesn't appear and we skip it.
Otherwise,
this set has a nonzero accumulation point $\zeta$ and a subsequence converges to it.
We assume that
$$\{x_k^{\ast},(i_k,j_k)=(1,3)\}\rightarrow\zeta\neq0$$
without loss of generality.
Note that
$\{\matr{Q}_{k-1},(i_k,j_k)=(1,3)\}$ has an accumulation point.
We assume that
$$\{\matr{Q}_{k-1},(i_k,j_k)=(1,3)\}\rightarrow\matr{Q}^{(\ell_1)}$$
without loss of generality.
Then $\matr{Q}^{(\ell_2)}=\matr{Q}^{(\ell_1)}\matr{G}^{(1,3,\arctan\zeta)}$ is another different accumulation point.
Note that $\tens{A}(\matr{Q}^{(\ell_1)})^{(1,2,3)}$ and $\tens{A}(\matr{Q}^{(\ell_2)})^{(1,2,3)}$ satisfy the conditions in Lemma \ref{lemma-3-dimension-p-2}.
We see that
$$\tenselem{A}(\matr{Q}^{(\ell_1)})_{112}=\tenselem{A}(\matr{Q}^{(\ell_1)})_{122}=
\tenselem{A}(\matr{Q}^{(\ell_2)})_{112}=\tenselem{A}(\matr{Q}^{(\ell_2)})_{122}=0.$$

Since $\matr{Q}_{\ast}$ satisfies \eqref{eq-condition-not-zero},
we see that path (a) is the only possible path.
Then all the accumulation points satisfy \eqref{eq-condition-not-zero}.
Note that $\matr{Q}_{\ast}$ satisfies \eqref{eq-condition-not-zero-2} and $\tenselem{A}(\matr{Q}^{(\ell_1)})_{iii}=\tenselem{A}(\matr{Q}^{(\ell_2)})_{iii}$ for $3\leq i\leq n$ in path (a).
All the accumulation points satisfy \eqref{eq-condition-not-zero-2}.\\
\textbf{Step 2.}
Since path (b) in Step 1 doesn't appear, we get that
\begin{equation}\label{eq-proof-tends-0}
\{x_k^{\ast},(i_k,j_k)\in\set{C}_{2}\}\rightarrow 0
\end{equation}
in JLROA-G.
Let $\mathcal{N}(\matr{Q}_{\ast},\eta)$ be the neighborhood of
$\matr{Q}_{\ast}=\matr{Q}^{(1)}$ in $\ON{n}$ with radius $\eta>0$ such that
there exist no other accumulation points in this neighborhood.
If pair $(i,j)\in\set{C}_2$ satisfies that
\begin{equation}\label{eq-condition-infinite}
\{\matr{Q}_{k-1}\in\mathcal{N}(\matr{Q}_{\ast},\eta),(i_k,j_k)=(i,j)\}\
\text{is infinite},
\end{equation}
then $\tens{A}(\matr{Q}_{\ast})^{(i,j)}$ satisfies the conditions in Lemma \ref{lemma-double-derivative} by condition \eqref{eq-condition-not-zero-2}.
Then $\tenselem{A}(\matr{Q}_{\ast})_{iii}(\tenselem{A}(\matr{Q}_{\ast})_{iii}-2\tenselem{A}(\matr{Q}_{\ast})_{ijj})\neq0$.
Let
$$\rho_1\eqdef\min|\tenselem{A}(\matr{Q}_{\ast})_{iii}(\tenselem{A}(\matr{Q}_{\ast})_{iii}-2\tenselem{A}(\matr{Q}_{\ast})_{ijj}) |$$
for all pairs $(i,j)\in\set{C}_{2}$ satisfying \eqref{eq-condition-infinite}.
Then $\rho_1>0$.
For other accumulation points,
we can similarly get $\rho_\ell$ for $1<\ell\leq N$.
Then
\begin{equation}\label{eq-rho}
\rho\eqdef\min\rho_\ell>0.
\end{equation}
\textbf{Step 3.}
Now we show that there exists $\kappa>0$ such that
\begin{align}\label{eq-inequality-C2}
|h_{k}(\theta_k^{*})-h_{k}(0)|\geq \kappa |h_{k}'(0)||\theta_k^{*}|
\end{align}
for all $(i_k,j_k)\in\set{C}_{2}$.
Let $\tens{W} = \tens{A}(\matr{Q}_{k-1})$.
Denote $(i,j) = (i_{k},j_{k})$.
Note that $|x_k^{\ast}|<+\infty$ when $k$ is large enough by \eqref{eq-proof-tends-0}.
Then by \eqref{eq-order-3-station} and \eqref{eq-proof-tends-0}, 
we have that
\begin{align*}
\frac{h_{k}'(0)}{x_k^{\ast}}=
\frac{6\tenselem{W}_{iii}\tenselem{W}_{iij}}{x_k^{\ast}}
=-6\tenselem{W}_{iii}[(2\tenselem{W}_{ijj}-\tenselem{W}_{iii})+(\tenselem{W}_{jjj}-2\tenselem{W}_{iij})x_k^{\ast}-\tenselem{W}_{ijj}{x_k^{\ast}}^2]
\end{align*}
have accumulation points in the set
$$\{-6\tenselem{A}(\matr{Q}^{(\ell)})_{iii}(2\tenselem{A}(\matr{Q}^{(\ell)})_{ijj}
-\tenselem{A}(\matr{Q}^{(\ell)})_{iii}),\ \text{pair $(i,j)$ satisfies \eqref{eq-condition-infinite}}, \ 1\leq\ell\leq N\}$$
when $k\in\NN$ with $(i_{k},j_{k})\in\set{C}_2$.
It follows from \eqref{eq-rho} that there exists
$\upsilon>0$ such that $|h_{k}'(0)|\geq\upsilon|x_{k}^{\ast}|$ when $k$ is large enough with $(i_{k},j_{k})\in\set{C}_2$..
Then we get \eqref{eq-inequality-C2} by Lemma \ref{lemma-weak-inequality}.\\
\textbf{Step 4.}
If $\{x_k^{\ast},(i_k,j_k)=(1,2)\in\set{C}_1\}$ is finite,
we skip it.
Otherwise,
by \cite[(27)]{LUC2017globally},
we know that
\begin{equation}\label{eq:lemma-G-inequality-3}
|h_{k}(\theta_k^{*})-h_{k}(0)| = |\frac{x_k^{*}h^{'}_{k}(0)}{2(1-{x_k^{*}}^2)}| \geq \frac{1}{2}|h_{k}'(0)||\theta_k^{*}|
\end{equation}
for all $(i_k,j_k)\in\set{C}_{1}$.
Let $\omega=\min\{\kappa,1/2\}>0$.
By \eqref{eq-inequality-C2} and \eqref{eq:lemma-G-inequality-3},
we get that
\begin{equation*}
|h_{k}(\theta_k^{*})-h_{k}(0)| \geq \omega|h_{k}'(0)||\theta_{\ast}| \geq \frac{\sqrt{2}}{2}\omega\varepsilon\|\ProjGrad{{f}}{\matr{Q}_{k-1}}\|\|\matr{Q}_{k}-\matr{Q}_{k-1}\|,
\end{equation*}
for all $k\in\NN$.
Then $\matr{Q}_{\ast}$ is the unique limit point by \cref{theorem-convegence-general}.
\end{proof}

\section{JLROA-GP algorithm and its convergence}\label{sect-Jacobi-GS}
\subsection{JLROA-GP algorithm}\label{subsec:gp}

To prove better theoretical convergence results, in this section, we further propose a proximal variant of JLROA-G algorithm, which is called \emph{JLROA-GP} algorithm.
Let $\gamma(\theta):[-\pi, \pi]\rightarrow\RR^{+}$ be a $C^2$ function satisfying that $\eta_1 \theta^2\leq\gamma(\theta)\leq \eta_2 \theta^2$ for two positive constants $\eta_1,\eta_2>0$. 
This new variant is shown in \Cref{alg:jacobi-GS}. 

\begin{algorithm}
\caption{JLROA-GP algorithm}\label{alg:jacobi-GS}
\begin{algorithmic}[1]
\STATE{{\bf Input:} $\tens{A}\in\text{symm}(\RR^{n\times \cdots\times n})$, $1\leq p\leq n$,
$0<\varepsilon\leq 2/n$, $\delta>0$, 
a starting point $\matr{Q}_{0}$.}
\STATE{{\bf Output:} Sequence of iterations $\{\matr{Q}_{k}\}_{k\ge1}$.} 
\FOR{$k=1,2,\ldots$ until a stopping criterion is satisfied}
\STATE Choose a pair $(i_k,j_k)$ satisfying the inequality \eqref{eq:pair_selection_gradient} at $\matr{Q}_{k-1}$.
\STATE Solve $\theta^{*}_{k}$ that maximizes
\begin{equation}\label{eq:func_h_gp}
\tilde{h}_k(\theta) = h_k(\theta)-\delta\gamma(\theta),
\end{equation}
where $h_k(\theta)$ is defined as in \eqref{definition-h}.
\STATE Set $\matr{U}_k \eqdef \Gmat{i_k}{j_k}{\theta^{*}_k}$, and update $\matr{Q}_k = \matr{Q}_{k-1} \matr{U}_k$.
\ENDFOR
\end{algorithmic}
\end{algorithm}

\subsection{Convergence analysis}

In this subsection, we mainly prove the following convergence result for JLROA-GP algorithm.

\begin{theorem}\rm\label{theorem-main-covergence-GS}
In JLROA-GP algorithm, the iterates $\{\matr{Q}_k\}_{k\geq 1}$ converge to a stationary point $\matr{Q}_{*}$.
\end{theorem}

To prove \Cref{theorem-main-covergence-GS}, we first need to show two lemmas. 
In fact, the proofs of these two lemmas are included in the proof of \cite[Theorem 6.2]{LUC2017globally}. 
We present them here for convenience.

\begin{lemma}\label{lem:gp-conver-01}
In JLROA-GP algorithm, there exists $\sigma_1>0$ such that
\begin{equation}\label{eq:gp-weak-01}
	h_k(\theta^{*}_{k}) - h_k(0) \geq \sigma_1|\theta^{*}_{k}|^2.
\end{equation}
\end{lemma}

\begin{proof}
Since
\begin{equation*}
h_k(\theta^{*}_{k})-h_k(0)-\delta\gamma(\theta^{*}_{k}) =
\tilde{h}_k(\theta^{*}_{k}) - \tilde{h}_k(0) \geq 0,
\end{equation*}
we get that 
\begin{equation}
h_k(\theta^{*}_{k})-h_k(0)\geq\delta\gamma(\theta^{*}_{k}) \geq \delta\eta_1 |\theta^{*}_{k}|^2.
\end{equation}
The proof is complete by setting $\sigma_1=\delta\eta_1$.
\end{proof}

\begin{lemma}\label{lem:gp-conver-02}
In JLROA-GP algorithm, there exists $\sigma_2>0$ such that
\begin{equation}\label{eq:gp-weak-02}
|\theta^{*}_{k}|\geq \sigma_2|h_k'(0)|. 
\end{equation}
\end{lemma}
\begin{proof}
Define 
$$\bar{h}(\theta,\matr{Q})=f(\matr{Q}\Gmat{i}{j}{\theta})-\delta\gamma(\theta),$$
for $\theta\in[-\pi,\pi]$, $\matr{Q}\in\ON{n}$ and $1\leq i<j\leq n$.
Let 
$$M \eqdef \max_{\matr{Q}\in\ON{n}, \theta\in[-\pi,\pi],1\leq i<j\leq n}\left|\frac{\partial^2 \bar{h}}{\partial \theta^2}(\theta,\matr{Q})\right|.$$
Then $M<+\infty$ since $f$ and $\gamma$ are both $C^2$ and $\ON{n}$ is compact.
Therefore, we have that
$$|h_k'(0)|=|\tilde{h}_k'(0)|=|\tilde{h}_k'(\theta_k^{*})-\tilde{h}_k'(0)|\leq M|\theta_k^{*}|,$$
for any $\matr{Q}_{k-1}\in\ON{n}$, $\theta_k^{*}\in[-\pi,\pi]$ and $1\leq i_k<j_k\leq n$.
The proof is complete by setting $\sigma_2=1/M$.
\end{proof}


\begin{proof}[Proof of \cref{theorem-main-covergence-GS}]

Note that $\|\matr{Q}_k-\matr{Q}_{k-1}\|\leq \sqrt{2}|\theta^{*}_{k}|$ by \cite[Eq. (25)]{LUC2017globally}. By Lemma \ref{lem:gp-conver-01}, Lemma \ref{lem:gp-conver-02} and the inequality \eqref{eq:pair_selection_gradient}, we have that 
\begin{align*}
f(\matr{Q}_{k})-f(\matr{Q}_{k-1}) &= h_k(\theta^{*}_{k})-h_k(0) \geq \sigma_1|\theta^{*}_{k}|^2 \geq \sigma_1\sigma_2|\theta^{*}_{k}||h_k'(0)|\\
&\geq\frac{\varepsilon}{\sqrt{2}}\sigma_1\sigma_2\|\matr{Q}_k-\matr{Q}_{k-1}\| \|\ProjGrad{f}{\matr{Q}_{k-1}}\|.
\end{align*}
By \cref{theorem-convegence-general}, we see that there exists $\matr{Q}_{*}\in\ON{n}$ such that $\matr{Q}_{k}\rightarrow\matr{Q}_{*}$. 
Now we show that $\matr{Q}_{*}$ is a stationary point. 
In fact, since $h_k(\theta^{*}_{k}) - h_k(0)\rightarrow 0$, we have $\theta_k^{*}\rightarrow 0$ by \eqref{eq:gp-weak-01}, and thus $h_k'(0)\rightarrow 0$ by \eqref{eq:gp-weak-02}.
It follows by the inequality \eqref{eq:pair_selection_gradient} that $\|\ProjGrad{f}{\matr{Q}_{k-1}}\|\rightarrow 0$, and thus $\ProjGrad{f}{\matr{Q}_{*}}=0$. The proof is complete. 
\end{proof}

\begin{remark}
Compared with the Jacobi-PC algorithm in \cite[Sec. 6]{LUC2017globally}, in this paper, we choose the gradient based pair selection rule, and prove \cref{theorem-main-covergence-GS} for a general cost function $f$ and a general function $\gamma(\theta)$. 
\end{remark}
\begin{remark}
We note that the elementary update with the proximal term becomes more difficult. \fin
If $(i_k,j_k)\in\set{C}_1$, a choice of $\gamma(\theta)$ satisfying the conditions is $\gamma(\theta)=2\sin^2(\theta)\cos^2(\theta)$, which was used in \cite{LUC2017globally}. 
This choice of $\gamma$ has advantage that the function remains $\pi/2$-periodic, and therefore the complexity of finding the modified update is the same as the one without the proximal term. 
However, for the case $(i_k,j_k)\in\set{C}_2$ by adding a proximal term, we destroy the quadratic structure of $h_k(\theta)$, and we need to find roots of a polynomial of order $2d$ (instead of order $d$).
\end{remark}

\section{Numerical experiments}\label{sect-experiment}

In this section, we make some experiments to compare the performance of JLROA algorithm with the LROAT and SLROAT algorihtms in \cite{chen2009tensor}, and Trust region algorithm by \emph{Manopt Toolbox} in \cite{JMLR:v15:boumal14a}.
When $p=1$, LROAT and SLROAT are exactly the HOPM and SHOPM algorithms in \cite{Lathauwer00:rank-1approximation,kofidis2002best}, respectively.
We use the cyclic ordering of JLROA-C algorithm for simplicity except Example \ref{example-5}.
The LROAT and SLROAT algorihtms are both initialized via HOSVD \cite{Lathauwer00:TensorSVD}, because we find they generally have better performance in this case. 

\begin{example}\label{example-1}\rm
We randomly generate 1000 tensors in $\text{symm}(\RR^{10\times 10\times 10})$,
and run JLROA and SLROAT algorithms for them.
Denote by $\textsc{JVal}$ and $\textsc{SVal}$ the final value of
\eqref{eq-cost-func-1} obtained by JLROA and SLROAT, respectively.
Set the following notations.\\
(i) $\textsc{NumG}:$ the number of cases that $\textsc{JVal}$ is greater than $\textsc{SVal}$;\\
(ii) $\textsc{NumS}:$ the number of cases that $\textsc{JVal}$ is smaller than $\textsc{SVal}$;\\
(iii) $\textsc{NumE}:$ the number of cases that $\textsc{JVal}$ is equal\footnote{the difference is smaller than 0.0001.} to $\textsc{SVal}$;\\
(iv) $\textsc{RatioG}:$ the average of $\textsc{JVal}/\textsc{SVal}$ when $\textsc{JVal}$ is greater than $\textsc{SVal}$;\\
(v) $\textsc{RatioS}:$ the average of $\textsc{JVal}/\textsc{SVal}$ when $\textsc{JVal}$ is smaller than $\textsc{SVal}$.\\
The results are shown in \cref{table-example-1} and \cref{figure-example-1}.
It can be seen that JLROA algorithm has better performance when $p>2$.
They always get the same result when $p=1$.

\begin{table}[h!]
  \centering
  \caption{}
  \label{table-example-1}
  \scalebox{0.9}{
  \begin{tabular}{l c c c c c}
  \toprule
       & $\textsc{NumG}$ & $\textsc{NumS}$ & $\textsc{NumE}$ & $\textsc{RatioG}$ & $\textsc{RatioS}$\\
  \midrule
  $p=1$       &  0 & 0  &  1000 & ---  & --- \\
   \midrule
  $p=2$       & 328 & 441 & 231 & 1.0023 & 0.9982\\
    \midrule
  $p=5$       & 747 & 246 & 7 & 1.0042 & 0.9985\\
    \midrule
  $p=8$       & 900 & 99 & 1 & 1.0044 & 0.9992 \\
   \midrule
  $p=10$       & 815  & 180 & 5  & 1.0039 & 0.9996 \\
  \bottomrule
  \end{tabular}}
\end{table}
\end{example}

\begin{figure}[tbhp]
\centering
\subfloat[p=2]{\includegraphics[width=0.5\textwidth]{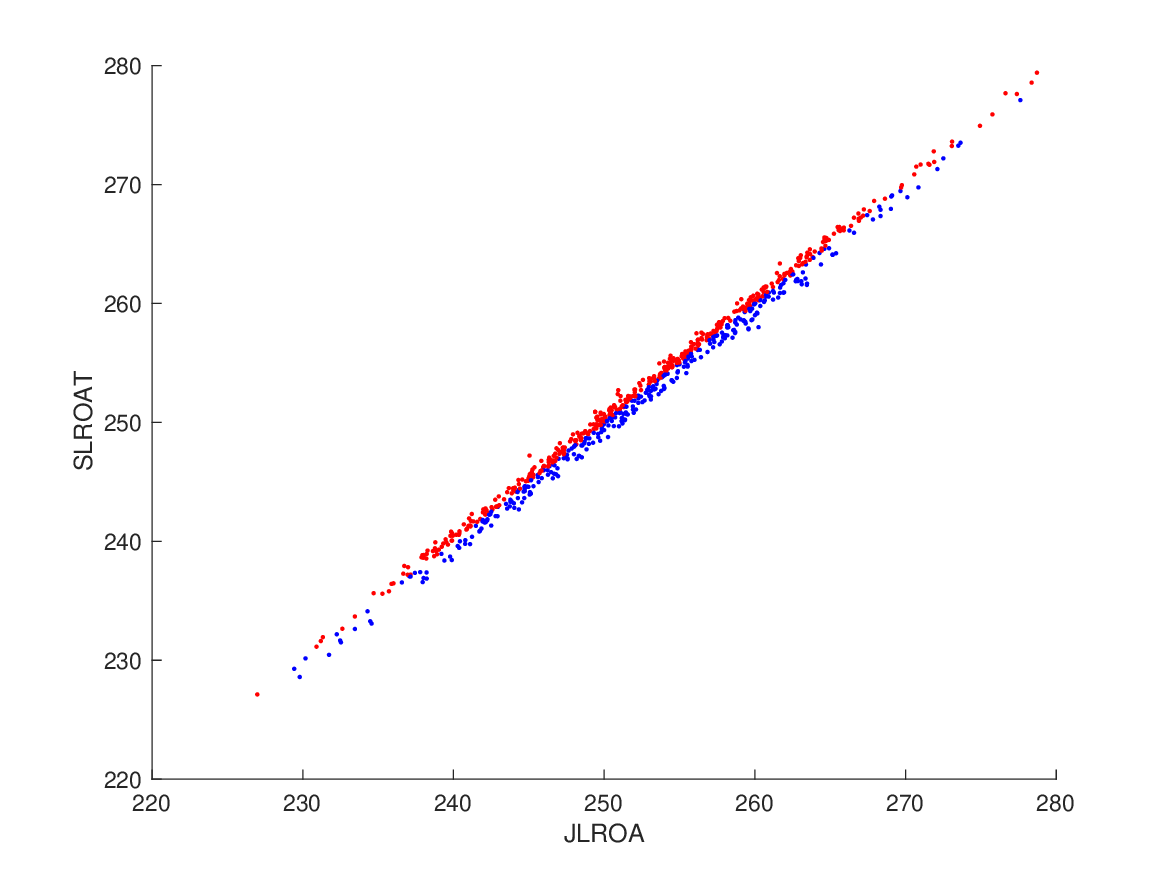}}\!\!\!
\subfloat[p=5]{\includegraphics[width=0.5\textwidth]{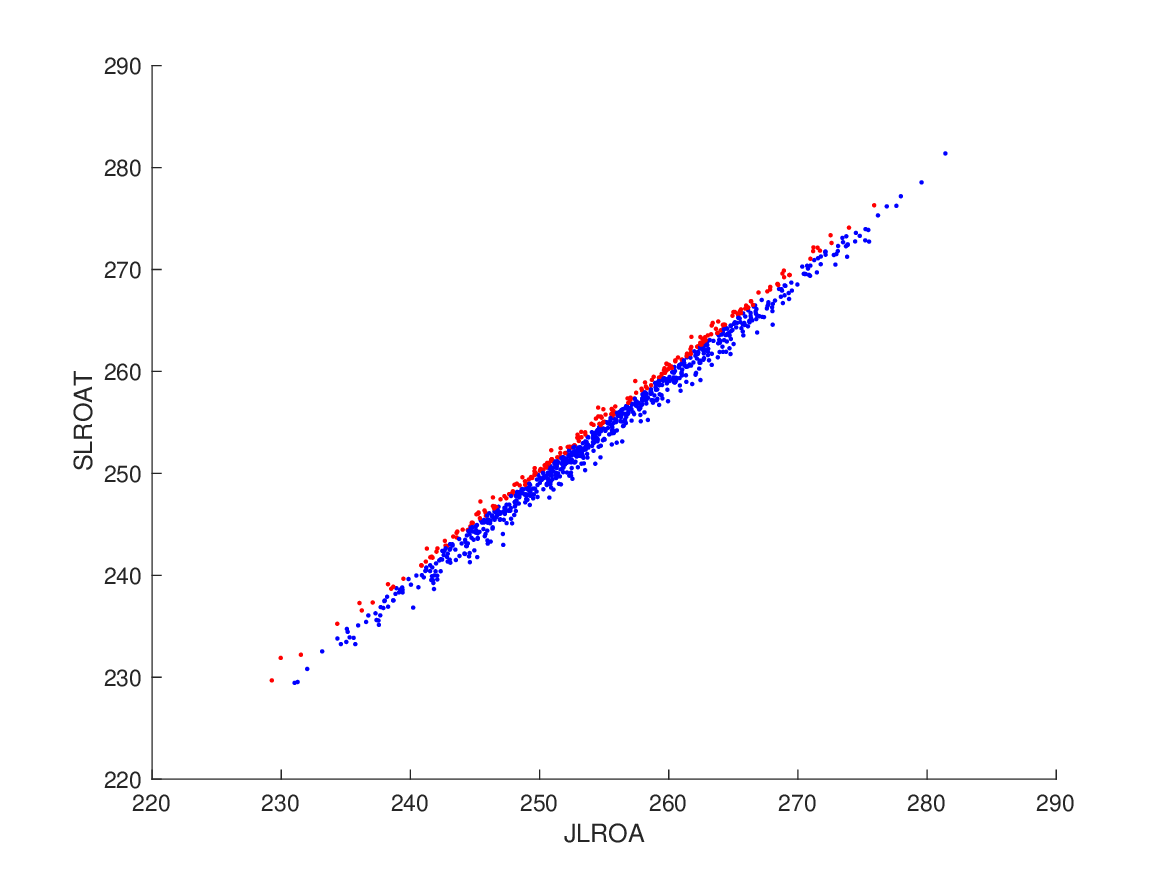}}\!\!\!
\subfloat[p=8]{\includegraphics[width=0.5\textwidth]{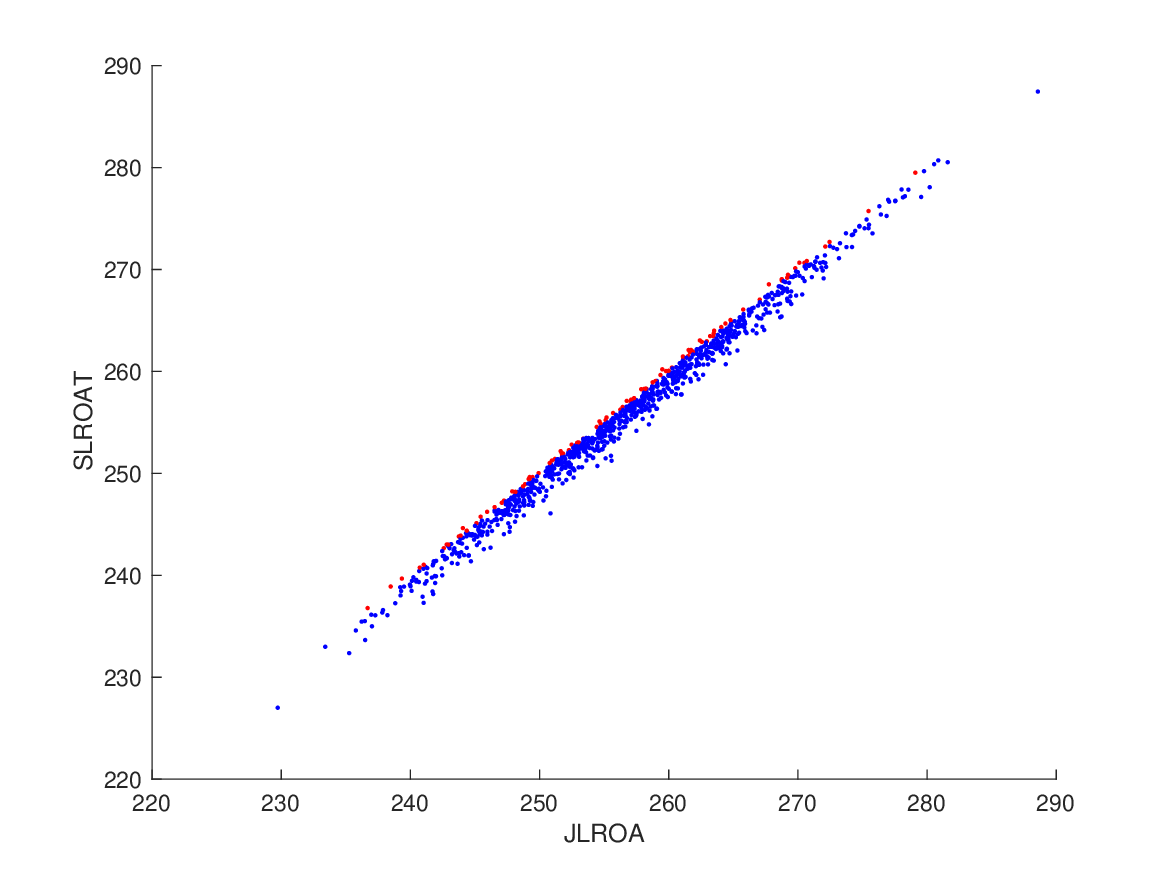}}\!\!\!
\subfloat[p=10]{\includegraphics[width=0.5\textwidth]{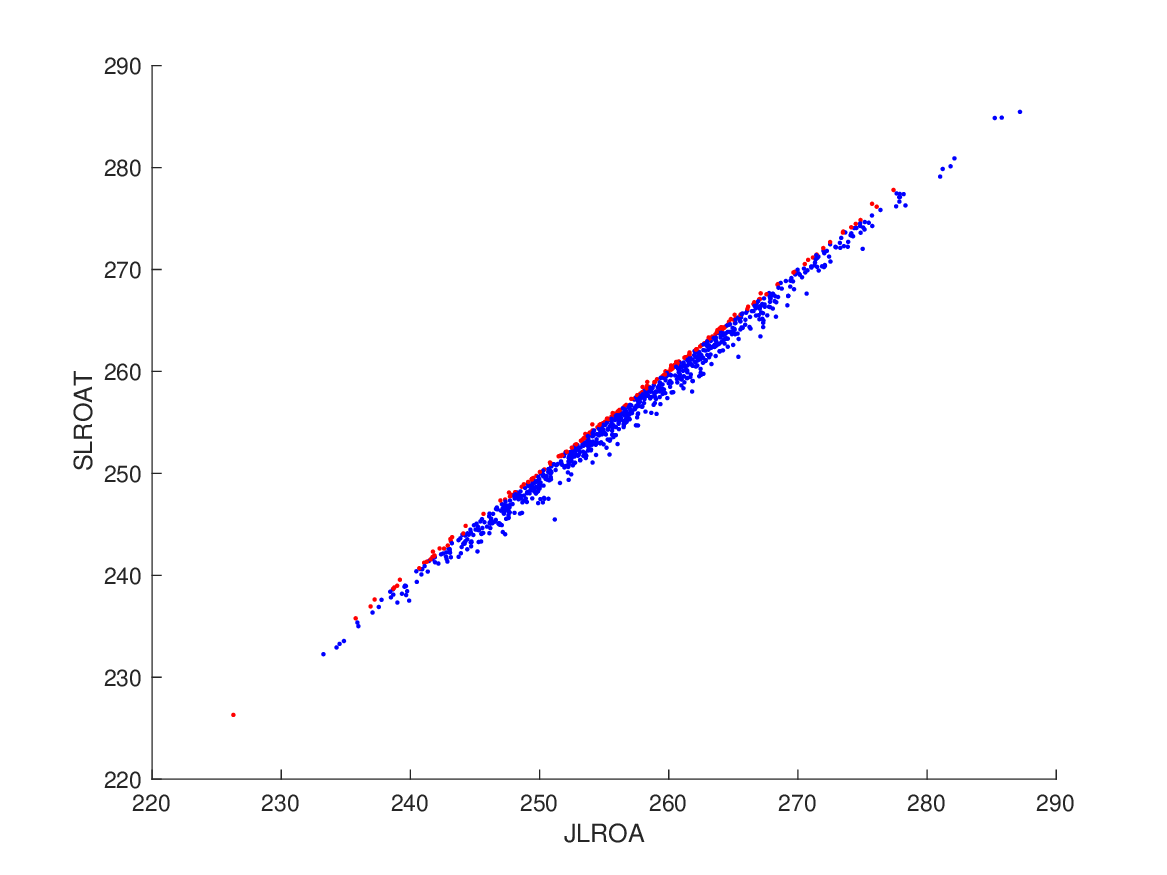}}
\caption{Distributions of points $(\textsc{JVal}, \textsc{SVal})$ in Example \ref{example-1}. The points are blue when $\textsc{JVal}$ is greater, and red when $\textsc{SVal}$ is greater.} 
\label{figure-example-1}
\end{figure}

\begin{example}\rm\label{example-2}
Let $\tens{A}\in\text{symm}(\RR^{3\times 3\times 3\times 3})$ such that
\begin{align*}
&\tenselem{A}_{1111} = 0.2883,\ \
\tenselem{A}_{1122} = -0.2485,\ \
\tenselem{A}_{1222} =  0.2972,\ \
\tenselem{A}_{1333} = -0.3619,\\
&\tenselem{A}_{2233} = 0.2127,\ \
\tenselem{A}_{1112} = -0.0031,\ \
\tenselem{A}_{1123} = -0.2939,\ \
\tenselem{A}_{1223} = 0.1862,\\
&\tenselem{A}_{2222} = 0.1241,\ \
\tenselem{A}_{2333} =  0.2727,\ \
\tenselem{A}_{1113} = 0.1973,\ \
\tenselem{A}_{1133} = 0.3847,\\
&\tenselem{A}_{1233} = 0.0919,\ \
\tenselem{A}_{2223} = -0.3420,\ \
\tenselem{A}_{3333} = -0.3054,\ \
\end{align*}
as in \cite[Example 1]{kofidis2002best} and \cite[Section 6.1]{chen2009tensor}.
It has been shown in \cite{kofidis2002best,chen2009tensor} that SHOPM ($p=1$) and SLROAT ($p=2$) fail to converge for $\tens{A}$.
We now see the convergence behaviour of JLROA algorithm.
The results of JLROA, SLROAT and LROAT algorithms are shown in \cref{figure-example-2}.
It can be seen that JLROA performances are always better than or equal to those of SLROAT and LROAT.
\end{example}

\begin{figure}[tbhp]
	\centering
	\subfloat[$p=1$]{\includegraphics[width=0.5\textwidth]{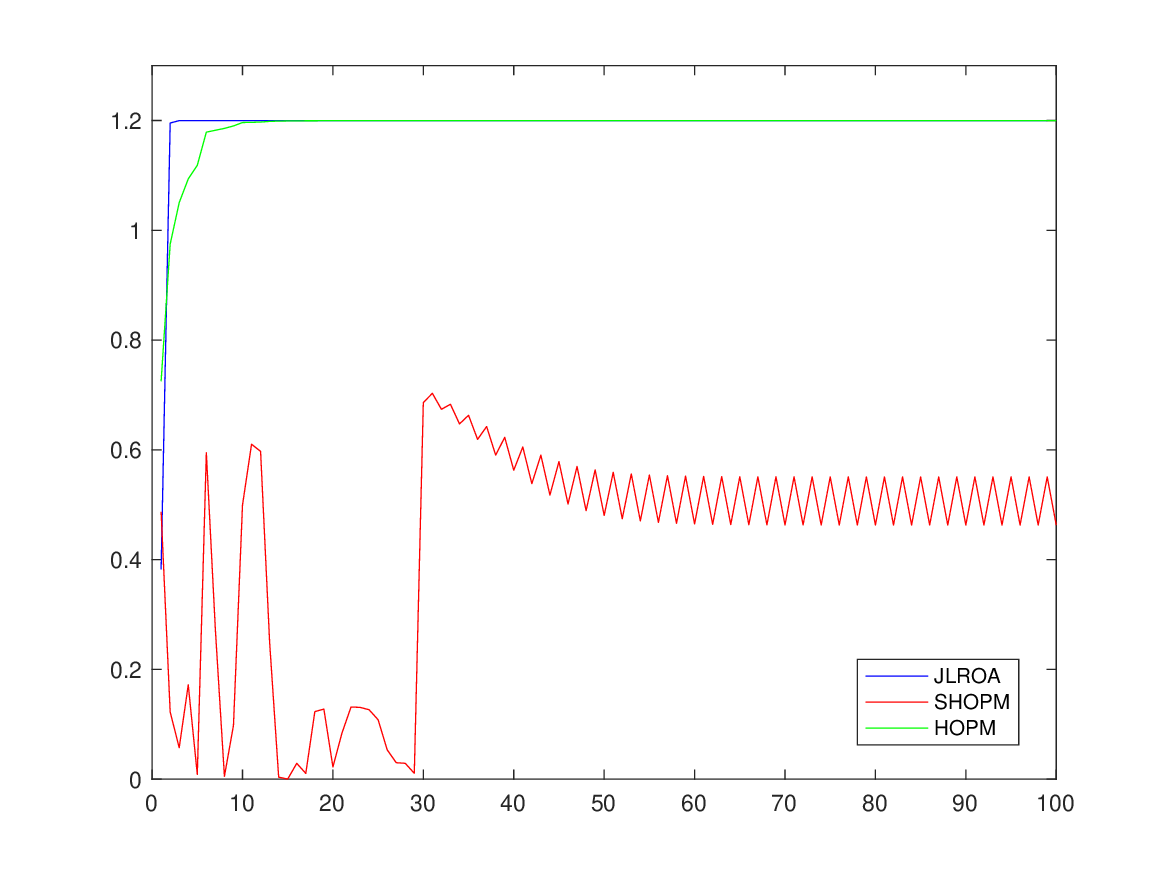}}\!\!\!
	\subfloat[$p=2$]{\includegraphics[width=0.5\textwidth]{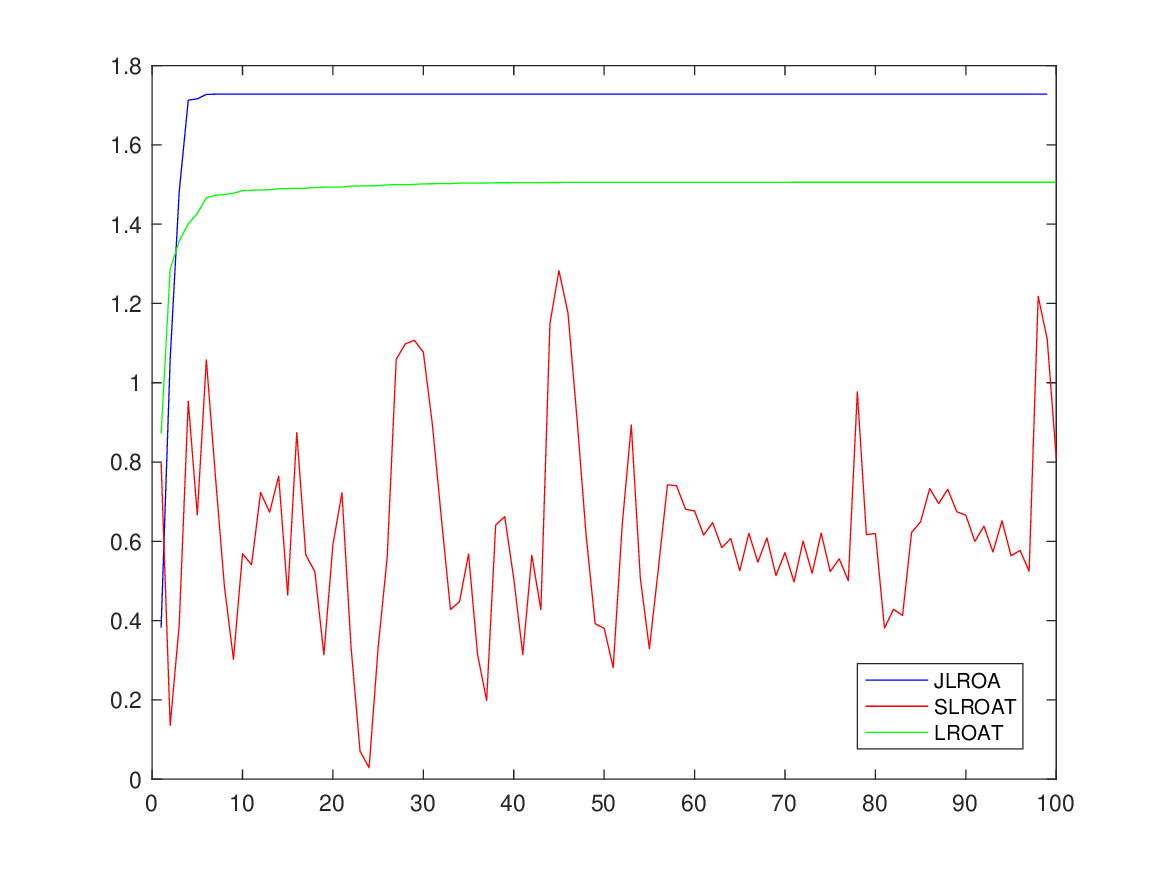}}\!\!\!
	\subfloat[$p=3$]{\includegraphics[width=0.5\textwidth]{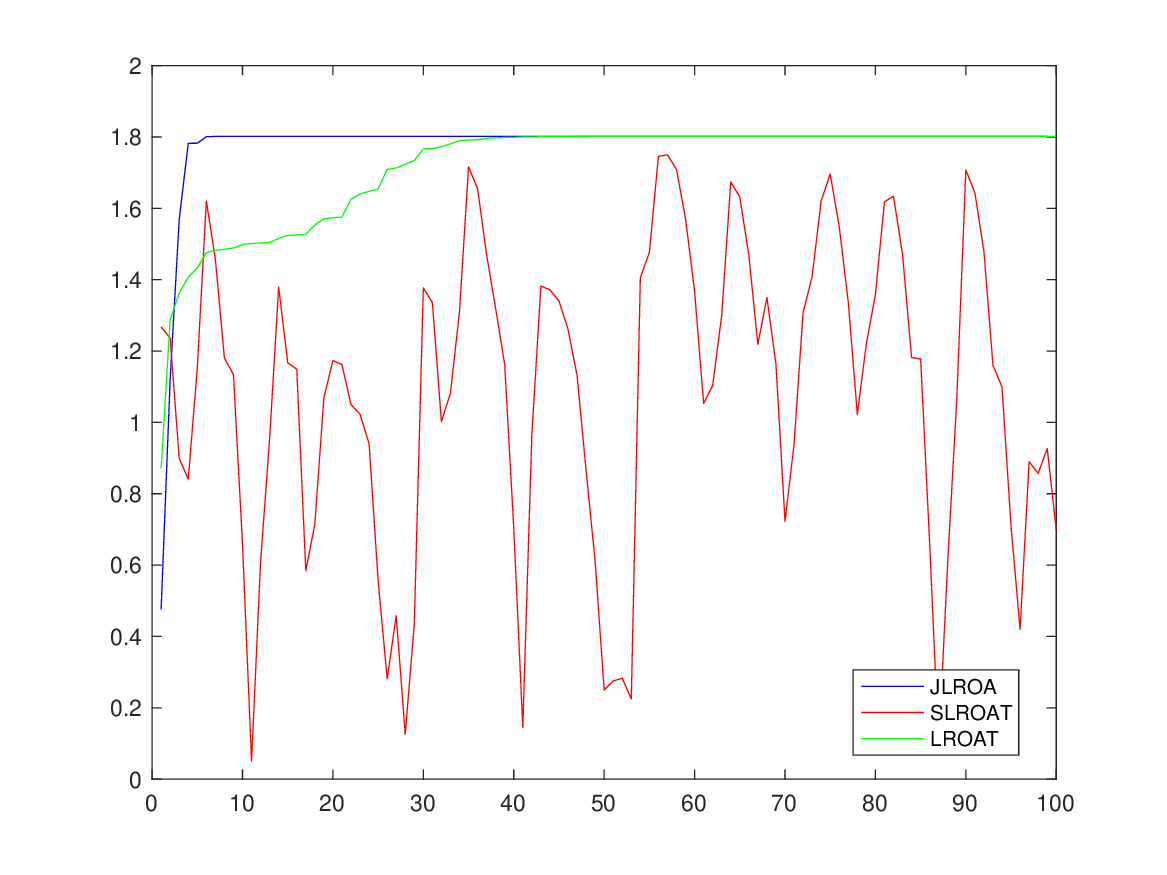}}
	\caption{Results of Example \ref{example-2}.} 
	\label{figure-example-2}
\end{figure}

\begin{example}\label{example-3}\rm
	We randomly generate 1000 tensors in $\text{symm}(\RR^{10\times 10\times 10})$,
	and run JLROA and Trust region algorithms for them.
	Denote by $\textsc{JVal}$ and $\textsc{TVal}$ the final value of
	\eqref{eq-cost-func-1} obtained by JLROA and Trust region, respectively.
	Set the following notations.\\
	(i) $\textsc{NumG}:$ the number of cases that $\textsc{JVal}$ is greater than $\textsc{TVal}$;\\
	(ii) $\textsc{NumS}:$ the number of cases that $\textsc{JVal}$ is smaller than $\textsc{TVal}$;\\
	(iii) $\textsc{NumE}:$ the number of cases that $\textsc{JVal}$ is equal\footnote{the difference is smaller than 0.0001.} to $\textsc{TVal}$;\\
	(iv) $\textsc{RatioG}:$ the average of $\textsc{JVal}/\textsc{TVal}$ when $\textsc{JVal}$ is greater than $\textsc{TVal}$;\\
	(v) $\textsc{RatioS}:$ the average of $\textsc{JVal}/\textsc{TVal}$ when $\textsc{JVal}$ is smaller than $\textsc{TVal}$.\\
	The results are shown in \cref{table-example-3} and \cref{figure-example-3}.
	It can be seen that $\textsc{RatioG}$ is very large when $p=1,2$, which means that Turst region is not so stable as JLROA in these two cases. 
	Correspondingly, Trust region algorithm has generally better performance when $p>2$.
	
	\begin{table}[h!]
		\centering
		\caption{}
		\label{table-example-3}
		\scalebox{0.9}{
			\begin{tabular}{l c c c c c}
				\toprule
				& $\textsc{NumG}$ & $\textsc{NumS}$ & $\textsc{NumE}$ & $\textsc{RatioG}$ & $\textsc{RatioS}$\\
				\midrule
				$p=1$       &  125 & 0  &  875 & 211.7822  & --- \\
				\midrule
				$p=2$       & 395 & 360 & 245 & 5.0299 & 0.9986\\
				\midrule
				$p=5$       & 431 & 555 & 14 & 1.0016 & 0.9987\\
				\midrule
				$p=8$       & 393 & 604 & 3 & 1.0011 & 0.9992 \\
				\midrule
				$p=10$       & 35  & 962 & 3  & 1.0002 & 0.9995 \\
				\bottomrule
		\end{tabular}}
	\end{table}
\end{example}

\begin{figure}[tbhp]
\centering
\subfloat[p=1]{\includegraphics[width=0.5\textwidth]{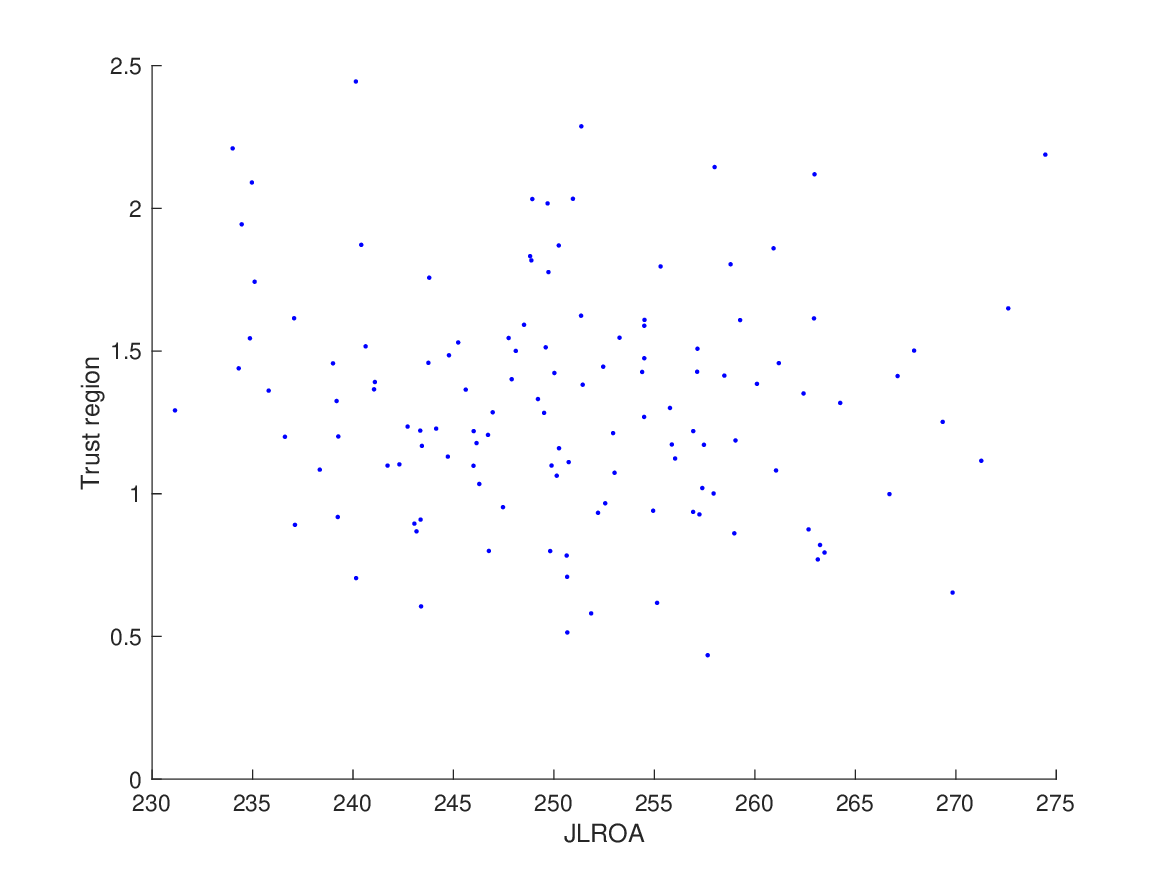}}\!\!\!
\subfloat[p=2]{\includegraphics[width=0.5\textwidth]{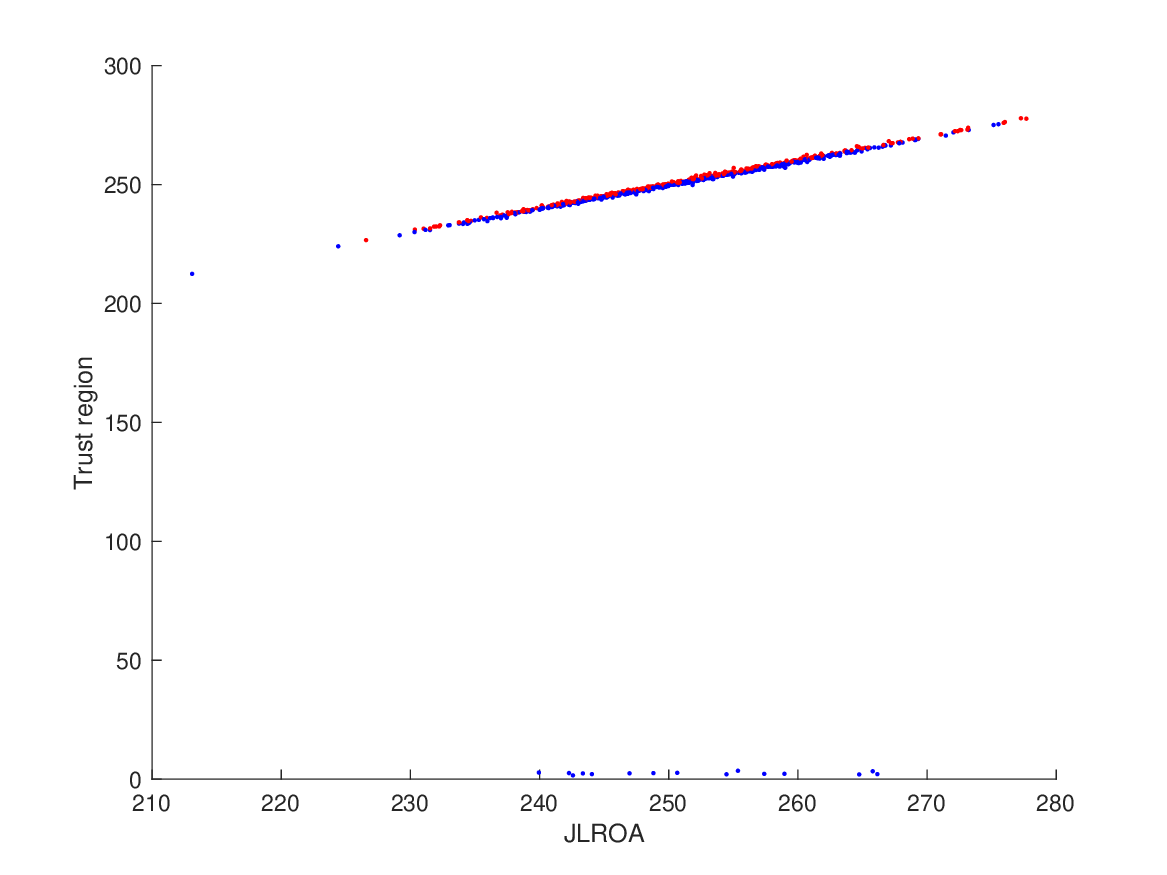}}\!\!\!
\subfloat[p=5]{\includegraphics[width=0.5\textwidth]{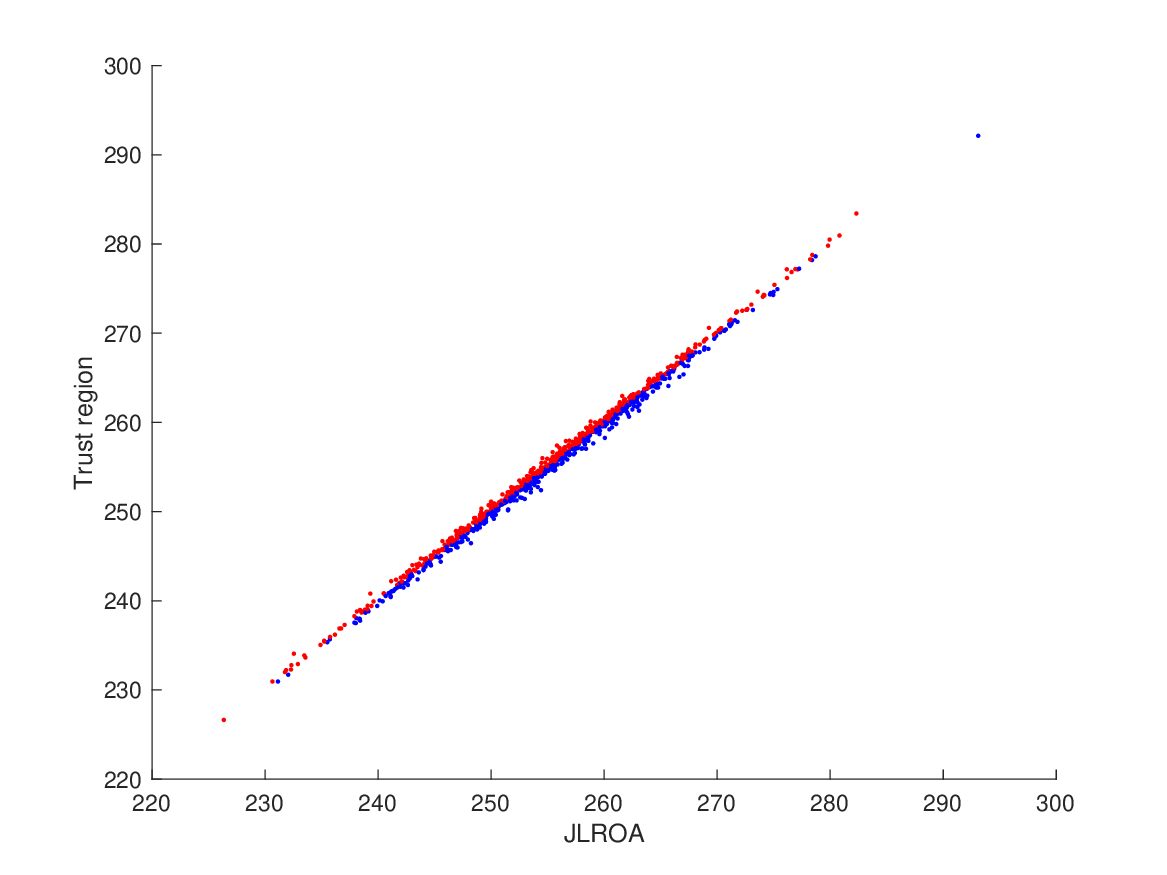}}\!\!\!
\subfloat[p=8]{\includegraphics[width=0.5\textwidth]{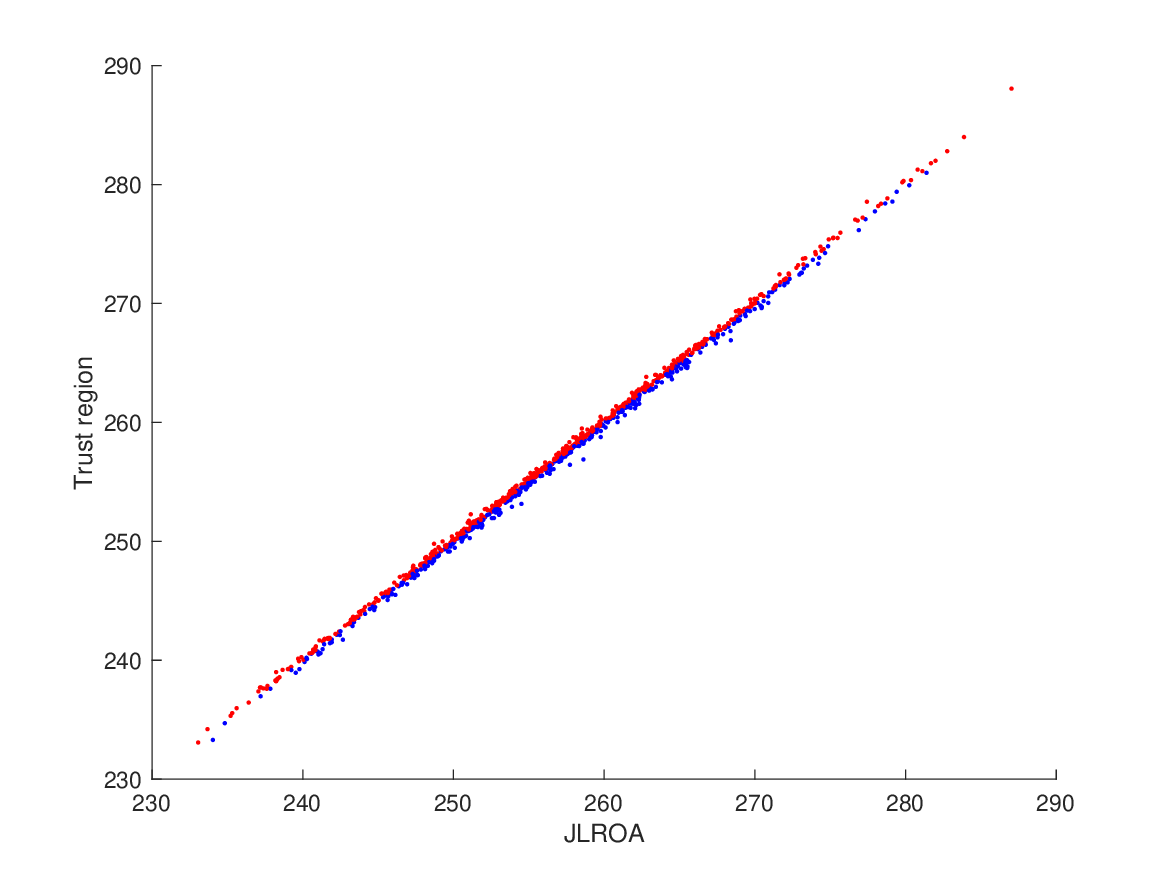}}
\caption{Distributions of points $(\textsc{JVal}, \textsc{TVal})$ in Example \ref{example-3}. The points are blue when $\textsc{JVal}$ is greater, and red when $\textsc{TVal}$ is greater.} 
\label{figure-example-3}
\end{figure}

%
%

\begin{example}\label{example-5}\rm
Let $\tens{A}\in\text{symm}(\RR^{10\times 10\times 10})$ and $p=2$.
Suppose that $\matr{Q}_{\ast}$ is an accumulation point of JLROA-G. 
To check the frequency of conditions \eqref{eq-condition-not-zero} and \eqref{eq-condition-not-zero-2} being satisfied, we define 
\begin{equation*}
\omega = \min\{|\tenselem{W}_{112}|, |\tenselem{W}_{122}|, |\tenselem{W}_{333}|, \cdots, |\tenselem{W}_{nnn}|\},
\end{equation*}
where $\tens{W}=\tens{A}(\matr{Q}_{\ast})$.
We choose the iteration $\matr{Q}_{K}$ as the approximation of an accumulation point when $K$ is large enough ($K=500$ in this experiment). 
We randomly generate $\tens{A}\in\text{symm}(\RR^{10\times 10\times 10})$ for $1000$ times,
and run JLROA-G to see the frequency that $\omega>0$ (greater than 0.0001).
The results are shown in \cref{figure-example-5},
where $\omega>0$ for $991$ times.
It can be seen that the conditions \eqref{eq-condition-not-zero} and \eqref{eq-condition-not-zero-2} are satisfied in most cases.

\end{example}

\begin{figure}[tbhp]
	\centering
	\includegraphics[width=0.7\textwidth]{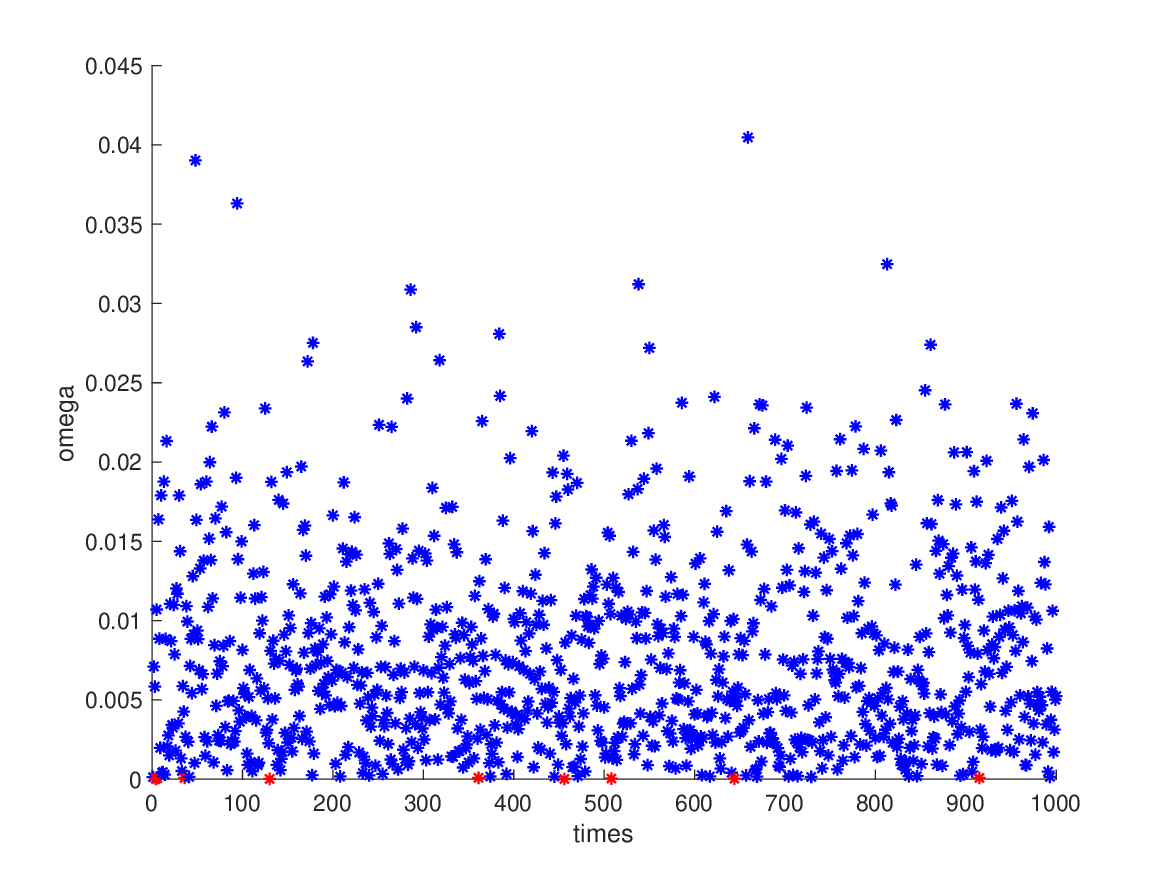}
	\caption{Results of Example \ref{example-5}. Blue points mean that $\omega>0$, while red points mean that $\omega=0$.} 
	\label{figure-example-5}
\end{figure}

\bibliographystyle{siamplain}
\bibliography{JLROA}

\end{document}